\documentclass[reqno, 12pt]{amsart}
\usepackage{mathrsfs}
\usepackage{amscd}
\usepackage{amsmath}
\usepackage{latexsym}
\usepackage{amsfonts}
\usepackage{amssymb}
\usepackage{amsthm}
\usepackage{graphicx}
\usepackage{xcolor}
\colorlet{RED}{red}
\usepackage{hyperref}
\usepackage{makecell}
\usepackage{array}
\usepackage{booktabs}   
\usepackage{multirow}
\usepackage[ruled,vlined]{algorithm2e}
\usepackage[section]{placeins}

\usepackage{tikz}
\usetikzlibrary{arrows.meta}

\newtheorem{theorem}{Theorem}[section]

\newtheorem{corollary}[theorem]{Corollary}
\newtheorem{lemma}[theorem]{Lemma}

\newtheorem{remark}[theorem]{Remark}

\numberwithin{equation}{section}

\title[Inverse moving point-like scatterer problem]{Detection of a moving point-like scatterer using moving receivers and a moving emitter}

\author[Minghui Li]{Minghui Li}
\author[Guanghui Hu]{Guanghui Hu}
\author[Hongyu Liu]{Hongyu Liu}
\address[Minghui Li]{School of Mathematical Sciences, Nankai University, Tianjin 300071, China}
\email{liminghui1122@outlook.com}

\address[Guanghui Hu]{School of Mathematical Sciences and LPMC, Nankai University, Tianjin 300071, China}
\email{ghhu@nankai.edu.cn}
\thanks{The work of G. Hu is supported by the National Natural Science Foundation of China (No. 12425112), the Fundamental Research Funds for Central Universities in China (No. 050-63263073) and the Natural Science Foundation of Tianjin (No. 25JCZDJC00970).}

\address[Hongyu Liu]{Department of Mathematics, City University of Hong Kong, Kowloon, Hong Kong}
\email{hongyliu@cityu.edu.hk}


\begin{document}
\begin{abstract}
We consider an inverse scattering problem for the scalar wave equation in which
the point emitter, a point-like scatterer, and multiple receivers are all moving. The goal is
to reconstruct the trajectory of a moving point-like scatterer from time-dependent measurements of the scattered field.
To this end, we obtain the moving point-like scatterer model by analyzing a suitably scaled short-range potential model.
By exploiting the retarded-time structure, we find that the trajectory of the moving point-like scatterer satisfies a coupled system of nonlinear ordinary differential equations determined by the measurement data. Based on this formulation, we propose a
reconstruction algorithm that combines initial localization with the solution of the derived
ODE system. Numerical experiments demonstrate that the method is accurate and robust
with respect to noise, with reconstruction errors remaining stable under moderate perturbations.
\end{abstract}

\maketitle

\section{Introduction}

Wave propagation modeled by the scalar wave equation underpins numerous applications, including radar, sonar, medical imaging~\cite{KianTriki2025}, and seismic exploration. Inverse scattering problems, which aim to recover unknown sources or scatterers from measured wave fields, have been extensively studied; the classical theory is summarized in the monograph \cite{ColtonKress}. In many realistic scenarios, the emitter, the scatterer, and the receivers may all be in motion, leading to significantly more challenging inverse problems. The use of moving receivers and emitters is particularly advantageous in dynamic environments, where fixed devices are impractical or insufficient: they offer enhanced coverage, adaptability to target motion, and the ability to collect diverse data from varying perspectives, which can improve reconstruction quality.

The identification of a moving point emitter from boundary or exterior measurements has received considerable attention. Early uniqueness results for recovering point emitters from time-dependent data were established by El Badia and Ha-Duong \cite{ElBadiaHaDuong}. Nakaguchi, Inui, and Ohnaka \cite{Nakaguchi2012} developed an algebraic reconstruction method using only a single observation point. More recently, Wang and Hu \cite{WangHu2026} proposed a frequency-domain approach to recover the orbit of a moving point emitter from multi-directional data, while Guo, Hu, and Ma \cite{Guo2023} considered multifrequency data measured at sparse observation directions. For the case of sparse partial measurements, Liu, Guo, and Sun \cite{Liu2021} introduced a deterministic-statistical approach. Al Jebawy, Elbadia, and Triki \cite{AlJebawy2022} investigated the inverse moving point source problem for the wave equation and derived Lipschitz stability estimates. Wang, Karamehmedovi$\acute{c}$ and Triki \cite{Wang2023} further developed a Bayesian framework that combines uniqueness, stability, and statistical inference for localizing moving sources. Related inverse source problems have also been studied for fixed point
sources with time-dependent amplitudes in parabolic equations
\cite{HuangJinKianTriki2026} and for wave equations using correlations
of passive measurements~\cite{TrikiLinderSteinleinKaramehmedovic2024}.
Other relevant contributions include the analysis of retarded potentials in electrodynamics and elastodynamics by Lazar \cite{Lazar2013}, super-resolution time-reversal focusing by Garnier and Fink \cite{Garnier2015}, and synthetic aperture imaging and wave propagation in random media by Borcea and co-authors \cite{Borcea2012,Borcea2017,Borcea2019}. Hu, Kian, Li, and Zhao \cite{Hu2018,Hu2019,Hu2020} studied inverse source problems in electrodynamics, including moving sources, and established uniqueness results in unbounded domains. Wang, Guo, Li, and Liu \cite{Wang2017} proposed a mathematical design of a novel input device using a moving acoustic emitter.

When the moving object is a scatterer rather than a source, the problem becomes nonlinear and significantly more difficult. Khoa and co-authors \cite{Klibanov2020a,Klibanov2020b} developed a convexification method based on Carleman estimates for a three-dimensional inverse scattering problem with a moving point emitter, proving Lipschitz stability and validating the method with experimental data. Sun, Chen, Gao, Li, and Sun \cite{Sun2025} considered inverse obstacle scattering with a single moving emitter and applied a direct sampling method to reconstruct both point-like and extended scatterers. These studies, however, assume that either the emitter or the scatterer moves, while the other components remain stationary.

A rigorous modeling of point-like scatterers can be achieved using the theory of point interactions arising from quantum mechancis \cite{Albeverio2005}, which is equivalent with the Foldy approach \cite{Foldy1945}.
Stationary point scatterers in the presence of an extended obstacle have been studied via self-adjoint extensions of the Dirichlet Laplacian; see, e.g., the work of Hu, Mantile, and Sini \cite{HuMantileSini2014}, which establishes a Krein-type resolvent formula and a factorization method for the inverse problem. In the time domain, wave equations with point interactions have been analyzed by Noja and Posilicano \cite{NojaPosilicano1998,NojaPosilicano2005} and Kurasov and Posilicano \cite{KurasovPosilicano2005}, who derived a general impedance-type boundary condition for the point-interaction model  and finite-speed propagation properties. In the Laplace domain, Mantile and Posilicano \cite{MantilePosilicano2020} further developed a factorization method for inverse wave scattering, providing an alternative framework that also accommodates point-interaction models for the wave equation. 

In this paper, we study an inverse scattering problem for the scalar wave equation where the point emitter, a point-like scatterer, and multiple receivers are all moving along prescribed trajectories. To the best of our knowledge, this fully dynamic configuration has not been rigorously addressed before. Our approach is inspired by the distance-function methodology introduced in \cite[Section 3]{Hu2020} and further developed for stability and numerical implementation in \cite{AlJebawy2022,LiHuZhao2025}; these works focus on moving point sources (linear problems). In contrast, our problem involves a moving point scatterer, which is inherently nonlinear. We extend the point-interaction model from the stationary setting \cite{HuMantileSini2014} to moving scatterers by adapting the time-domain formulations of \cite{NojaPosilicano1998,NojaPosilicano2005,KurasovPosilicano2005}.  
Specifically, we rigorously derive the moving point-scatterer model as a limit of suitably scaled short-range potentials which are concentrated near a moving trajectory. Based on this model, we obtain a
nonlinear ODE for the scatterer trajectory from the measured signals at four
moving receivers, with the velocity and scattering strength determined by
linear systems. The scattering parameter is then recovered from the response
equation. Numerical experiments demonstrate the accuracy and robustness of
the proposed method under various noise levels and dynamic configurations.

The remainder of the paper is organized as follows. In Section~\ref{sec:problem-formulation}, we describe preliminary assumptions on the moving point scatterer, the moving emitter and receivers, and formulate the inverse problem.
In Section~\ref{sec:point-model-limit}, we derive the moving point-scatterer model
as a limit of suitably scaled short-range potentials.
Section~\ref{sec:inversion-algorithm} presents the reconstruction algorithm,
which uses first-arrival times for initialization and recovers the trajectory
and scattering strength through an ODE and linear systems.
Section~\ref{sec:numerical-examples} presents numerical experiments
illustrating the accuracy and noise robustness of the method.
Finally, Section~\ref{sec:conclusion} summarizes the results and discusses
possible extensions.

\section{Problem formulation}\label{sec:problem-formulation}

In this paper, we consider a three-dimensional acoustic scattering problem from a moving point-like target, where both the emitter and receivers are also allowed to move along prescribed trajectories. Specifically, a moving point emitter with strength \(\varphi(t)\) travels along \(\boldsymbol{z}(t)\), generating an incident field \(u^{\mathrm{in}}\). A moving point-like scatterer follows \(\boldsymbol{y}(t)\) and produces a scattered field \(u^{\mathrm{sc}}\) that is recorded by four receivers moving along trajectories \(\boldsymbol{x}_i(t)\), \(i=1,2,3,4\). The four receivers satisfy the noncoplanarity conditions stated below. All speeds are bounded by the wave speed \(c\). A schematic illustration of this configuration is shown in Figure~\ref{fig:scattering-setup}.

We impose the following assumptions throughout the paper:
\begin{enumerate}
	\item\label{ass:emitter-strength} The emitter strength satisfies \(\varphi\in C^1([0,\infty))\) and \(\varphi(t)>0\) for all \(t\ge0\).
	\item\label{ass:subsonic-trajectories} (Subsonic condition) The trajectories \(\boldsymbol{x}_1,\boldsymbol{x}_2,\boldsymbol{x}_3,\boldsymbol{x}_4,\boldsymbol y,\boldsymbol z\) are \(C^2\)-smooth. Their speeds are all strictly bounded by the wave speed $c$: there exists \(c_0\in(0,c)\) such that
	\[
	|\boldsymbol{z}'(t)| < c_0,\quad |\boldsymbol{y}'(t)| < c_0,\quad |\boldsymbol{x}_i'(t)| < c_0,\qquad \text{for all}\, t>0,\ i=1,2,3,4.
	\]
	\item\label{ass:trajectory-separation} (Separation condition) There exist two smooth bounded domains \(\Omega_1\subset\Omega_2\) such that
	\[
	\boldsymbol{y}(t)\in\Omega_1,\quad \boldsymbol{z}(t),\boldsymbol{x}_i(t)\in\mathbb{R}^3\setminus\overline{\Omega_2}\quad \text{for all}\;t\ge0,\ i=1,2,3,4.
	\]
	\item\label{ass:receiver-position-noncoplanarity} For any \(t_1,t_2,t_3,t_4>0\), the receiver positions \(\boldsymbol{x}_i(t_i)\), \(i=1,2,3,4\), are noncoplanar. 
	\item\label{ass:receiver-direction-noncoplanarity} For any \(t_1,t_2,t_3,t_4>0\) and \(\boldsymbol{\xi}\in\overline{\Omega_1}\), the four points
	\[
	\frac{\boldsymbol{x}_i(t_i)-\boldsymbol{\xi}}
	{|\boldsymbol{x}_i(t_i)-\boldsymbol{\xi}|},
	\qquad i=1,2,3,4,
	\]
	are noncoplanar.
\end{enumerate}

\begin{figure}[!ht]
	\centering
	\begin{tikzpicture}[scale=1.2, >=Stealth]
		\tikzset{>={Stealth[length=6pt,width=6pt]}}
		
		\draw[thick] (0,0) circle (2);
		
		\node at (-2.3,0) {$\partial\Omega_2$};

		\draw[thick, dashed] (0,0) circle (1.8);
		
		\node at (-1.4,0) {$\partial\Omega_1$};

		\draw[->, thick, blue] 
		(-0.8,-0.8) .. controls (-0.5,0.5) .. (0.3,0.2)
		.. controls (0.8,0.5) .. (0.6,1.0);
		\filldraw[blue] (0.3,0.2) circle (2pt);
		\node[blue] at (0,1.0) {$\boldsymbol{y}(t)$};
		\node[blue] at (-0.2,-1.3) {\small Scatterer};
		
		\draw[->, thick, red] 
		(-4,2) .. controls (-3,3) .. (-2,2.5)
		.. controls (-1.5,2.2) .. (-0.5,2.8);
		
		\filldraw[red] (-2,2.5) circle (2pt);
		\node[red] at (-1.4,2.8) {$\boldsymbol{z}(t)$};
		
		\node[red] at (-3.2,3.2) {\small Emitter};

		\draw[->, thick, green!60!black] 
		(3.5,-1.5) .. controls (3,-0.5) .. (2.5,0.5);
		\filldraw[green!60!black] (3,-0.5) circle (2pt);
		\node[green!60!black] at (2.5,-1.4) {$\boldsymbol{x}_1(t)$};
		
		\draw[->, thick, green!60!black] 
		(2.2,1.2) .. controls (3,1.8) .. (3.8,1.5);
		\filldraw[green!60!black] (2.9,1.65) circle (2pt);
		\node[green!60!black] at (3.5,1.9) {$\boldsymbol{x}_2(t)$};
		\node[green!60!black] at (3.5,2.5) {\small Receivers};

		\draw[->, dashed, red] (-2,2.5) -- (0.3,0.2);
		\node[red] at (-1,2) {\small $u^{\mathrm{in}}$};
		
		\draw[->, dashed, blue] (0.3,0.2) -- (3,-0.5);
		\draw[->, dashed, blue] (0.3,0.2) -- (2.9,1.65);
		\node[blue] at (1.3,0.4) {\small $u^{\mathrm{sc}}$};

		\draw[->, thick, green!60!black] 
		(-3.5,-1.8) .. controls (-2.8,-1.0) .. (-2.2,-0.6);
		\filldraw[green!60!black] (-2.8,-1.0) circle (2pt);
		\node[green!60!black] at (-3.6,-2.2) {$\boldsymbol{x}_3(t)$};

		\draw[->, thick, green!60!black] 
		(0,-4) .. controls (0.8,-3.2) .. (1.5,-2.6);
		\filldraw[green!60!black] (0.8,-3.2) circle (2pt);
		\node[green!60!black] at (0,-4.4) {$\boldsymbol{x}_4(t)$};
		
		\draw[->, dashed, blue] (0.3,0.2) -- (-2.8,-1.0);
		\draw[->, dashed, blue] (0.3,0.2) -- (0.8,-3.2);
		
	\end{tikzpicture}
	\caption{Wave scattering by a moving scatterer with moving emitter and receivers}
	\label{fig:scattering-setup}
\end{figure}
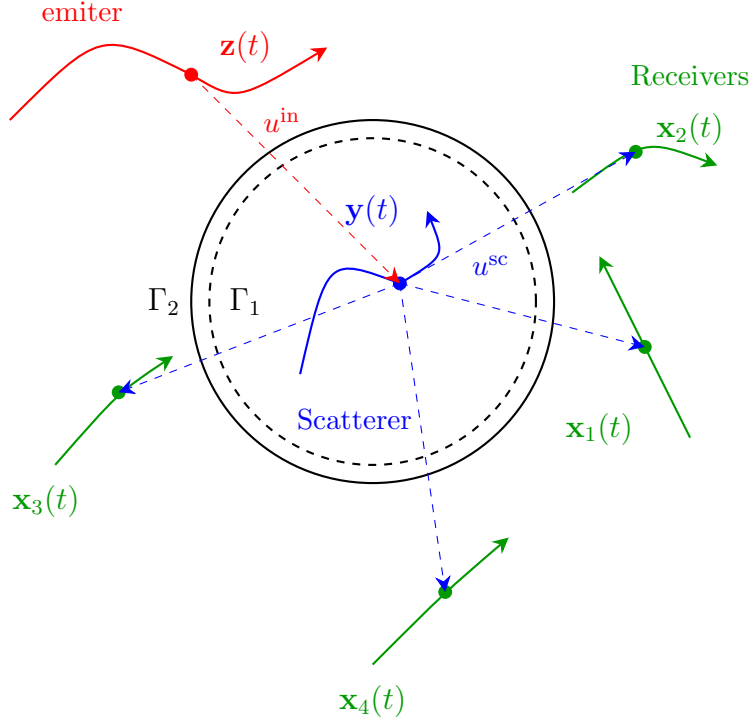

In Section~\ref{sec:point-model-limit}, we derive the moving
point-scatterer model as a limit of short-range potentials and
obtain the evolution equation for its scattering strength.
We first give explicit expressions for the incident and scattered fields and then formulate the inverse problem.

\subsection{Incident field from a moving point emitter}\label{sec:incident-field}
Suppose that the background medium is homogeneous and isotropic with a constant acoustic speed $c>0$.
The incident field \(u^{\mathrm{in}}\) generated by 
a moving point emitter then 
satisfies the inhomogeneous wave equation
\[
\begin{cases}
	(\frac{1}{c^2}\partial_{tt}-\Delta)u^{\mathrm{in}}(\boldsymbol{x},t)=\varphi(t)\,\delta(\boldsymbol{x}-\boldsymbol{z}(t)), & \boldsymbol{x}\in\mathbb{R}^3,\ t\ge 0,\\[4pt]
	u^{\mathrm{in}}(\boldsymbol{x},0)=\partial_t u^{\mathrm{in}}(\boldsymbol{x},0)=0, & \boldsymbol{x}\in\mathbb{R}^3.
\end{cases}
\]

The solution is obtained by convolving the source with the free-space fundamental solution, as in the derivation of the explicit solution in~\cite{LiHuZhao2025}.
Let $E(\boldsymbol{x},t)=\delta(t-|\boldsymbol{x}|/c)/(4\pi|\boldsymbol{x}|)$
be the retarded fundamental solution of $L=c^{-2}\partial_t^2-\Delta$, and let \(H\) denote the Heaviside step function.
Convolution with the source gives
\begin{align}
	u^{\mathrm{in}}(\boldsymbol x,t)
	&=\int_0^t\int_{\mathbb R^3}E(\boldsymbol x-\boldsymbol\eta,t-s)
	\varphi(s)\delta(\boldsymbol\eta-\boldsymbol z(s))\,d\boldsymbol\eta\,ds\notag\\
	&=\int_0^t\frac{\varphi(s)}{4\pi|\boldsymbol x-\boldsymbol z(s)|}
	\delta\!\left(t-s-\frac{|\boldsymbol x-\boldsymbol z(s)|}{c}\right)ds\notag\\
	&=\frac{H\bigl(t-|\boldsymbol x-\boldsymbol z(0)|/c\bigr)}
	{4\pi R_{\boldsymbol z}(\boldsymbol x,t)J_{\boldsymbol z}(\boldsymbol x,t)}
	\,\varphi\bigl(\tau_1(\boldsymbol x,t)\bigr),
	\qquad\boldsymbol x\ne\boldsymbol z(t),\label{eq:incident-field}
\end{align}
where the retarded time \(\tau_1(\boldsymbol{x},t)\in[0,t]\) is the unique solution of \(t - \tau_1 - |\boldsymbol{x} - \boldsymbol{z}(\tau_1)|/c = 0\), and
\[
R_{\boldsymbol z}(\boldsymbol x,t)
:=|\boldsymbol x-\boldsymbol z(\tau_1(\boldsymbol x,t))|,
\qquad
J_{\boldsymbol z}(\boldsymbol x,t)
:=1-\frac{\boldsymbol z'(\tau_1(\boldsymbol x,t))\cdot
	(\boldsymbol x-\boldsymbol z(\tau_1(\boldsymbol x,t)))}
{cR_{\boldsymbol z}(\boldsymbol x,t)}.
\]
The subsonic assumption ensures the existence and uniqueness of \(\tau_1\) and gives
\[
J_{\boldsymbol z}(\boldsymbol x,t)\geq 1-c_0/c>0.
\]

\subsection{Scattered field and recorded signals}

Using the point-interaction model established in Section~\ref{sec:point-model-limit}, we can decompose the total field into \(u = u^{\mathrm{in}} + u^{\mathrm{sc}}\), where the scattered field \(u^{\mathrm{sc}}\) produced by the moving point-like scatterer satisfies
\[
\begin{cases}(\frac{1}{c^2}\partial_{tt}-\Delta)u^{\mathrm{sc}}(\boldsymbol{x},t)=q(t)\,\delta(\boldsymbol{x}-\boldsymbol{y}(t)), & \boldsymbol{x}\in\mathbb{R}^3,\ t\ge 0,\\[4pt]u^{\mathrm{sc}}(\boldsymbol{x},0)=\partial_t u^{\mathrm{sc}}(\boldsymbol{x},0)=0, & \boldsymbol{x}\in\mathbb{R}^3,\end{cases}
\]
with the scattering strength \(q(t)\) satisfying
\begin{equation}\label{eq:response-law}
\frac{\gamma(t)^2}{4\pi c}q'(t)
+\left(\alpha(t)+\frac{\gamma(t)^4\boldsymbol y'(t)\cdot\boldsymbol y''(t)}{4\pi c^3}\right)q(t)
=u^{\mathrm{in}}(\boldsymbol y(t),t),\qquad q(0)=0,
\end{equation}
where $\gamma(t)=(1-|\boldsymbol y'(t)|^2/c^2)^{-1/2}$ and
\(\alpha\in C([0,\infty))\), with \(\alpha(t)>0\) for all \(t\geq0\), is the time-dependent scattering parameter of the point scatterer.

As in~\eqref{eq:incident-field}, the scattered field has the retarded form
\begin{equation}\label{eq:scattered-field}
	u^{\mathrm{sc}}(\boldsymbol{x},t) = \frac{H\bigl(t-|\boldsymbol{x}-\boldsymbol{y}(0)|/c\bigr)}{4\pi R_{\boldsymbol y}(\boldsymbol x,t)J_{\boldsymbol y}(\boldsymbol x,t)}\, q\bigl(\tau_2(\boldsymbol{x},t)\bigr),\qquad \boldsymbol{x}\neq\boldsymbol{y}(t),
\end{equation}
where the retarded time \(\tau_2(\boldsymbol{x},t)\in[0,t]\) is the unique solution of \(t-\tau_2-|\boldsymbol{x}-\boldsymbol{y}(\tau_2)|/c=0\), and
\[
R_{\boldsymbol y}(\boldsymbol x,t)
:=|\boldsymbol x-\boldsymbol y(\tau_2(\boldsymbol x,t))|,
\qquad
J_{\boldsymbol y}(\boldsymbol x,t)
:=1-\frac{\boldsymbol y'(\tau_2(\boldsymbol x,t))\cdot
	(\boldsymbol x-\boldsymbol y(\tau_2(\boldsymbol x,t)))}
{cR_{\boldsymbol y}(\boldsymbol x,t)}.
\]
The subsonic assumption also ensures the existence and uniqueness of \(\tau_2\) and gives
\[
J_{\boldsymbol y}(\boldsymbol x,t)\geq 1-c_0/c>0.
\]
Substituting the following expression for $q(t)$ into~\eqref{eq:scattered-field},
\begin{equation}\label{eq:scattering-strength}
q(t)=\frac{4\pi c}{\gamma(t)}\int_0^t
\exp\!\left(-4\pi c\int_s^t\alpha(r)\gamma(r)^{-2}\,dr\right)
\frac{u^{\mathrm{in}}(\boldsymbol y(s),s)}{\gamma(s)}\,ds,\qquad t>0,
\end{equation}
we arrive at
\begin{equation}\label{eq:scattered-field-explicit}
\begin{aligned}
u^{\mathrm{sc}}(\boldsymbol{x},t)
&= \frac{c\,H\bigl(t-|\boldsymbol{x}-\boldsymbol{y}(0)|/c\bigr)}
{R_{\boldsymbol y}(\boldsymbol x,t)J_{\boldsymbol y}(\boldsymbol x,t)
\gamma(\tau_2(\boldsymbol x,t))}\\
&\quad\times\int_0^{\tau_2(\boldsymbol{x},t)}
\exp\!\left(-4\pi c\int_s^{\tau_2(\boldsymbol x,t)}\alpha(r)\gamma(r)^{-2}\,dr\right)
\frac{u^{\mathrm{in}}(\boldsymbol{y}(s),s)}{\gamma(s)}\, ds,
\qquad \boldsymbol{x}\neq\boldsymbol{y}(t).
\end{aligned}
\end{equation}

For a stationary scatterer $\boldsymbol y(t)\equiv\boldsymbol y_0$,
\eqref{eq:response-law} reduces to
\begin{equation}\label{eq:stationary-response-law}
q'(t)+4\pi c\alpha(t)q(t)=4\pi c u^{\mathrm{in}}(\boldsymbol y_0,t),
\qquad q(0)=0.
\end{equation}
Writing $r_0=|\boldsymbol x-\boldsymbol y_0|>0$, we have
$R_{\boldsymbol y}=r_0$, $J_{\boldsymbol y}=1$, and
$\tau_2=t-r_0/c$ when \(t\ge r_0/c\). Thus
\eqref{eq:scattered-field-explicit} gives
\begin{equation}\label{eq:stationary-scattered-field}
u^{\mathrm{sc}}(\boldsymbol x,t)=
\begin{cases}
\displaystyle\frac{c}{r_0}\int_0^{t-r_0/c}
\exp\!\left(-4\pi c\int_s^{t-r_0/c}\alpha(r)\,dr\right)
u^{\mathrm{in}}(\boldsymbol y_0,s)\,ds,
& t\ge r_0/c,\\[6pt]
0,&t<r_0/c.
\end{cases}
\end{equation}
For constant $\alpha$ and $c=1$,
\eqref{eq:stationary-response-law} and~\eqref{eq:stationary-scattered-field}
coincide with the single-center linear point-interaction model in
\cite[Theorem~3.1]{NojaPosilicano2005}.

For the inverse problem, we denote the scattered fields measured at
the receivers moving along trajectories \(\boldsymbol{x}_i(t)\), \(i=1,2,3,4\), by
\[
U_{\boldsymbol{x}_i}(t) := u^{\mathrm{sc}}(\boldsymbol{x}_i(t),t).
\]
In this work we study the following inverse problem.

\noindent
\textbf{Inverse problem.} Given the scattered field measurements \(\{U_{\boldsymbol{x}_i}(t) \mid t\ge 0,\ i=1,2,3,4\}\), 
the emitter strength \(\varphi(t)\) and trajectory \(\boldsymbol{z}(t)\), and the receiver trajectories \(\boldsymbol{x}_i(t)\), we want to recover the unknown trajectory \(\boldsymbol{y}(t)\) and the scattering parameter \(\alpha(t)\) of the moving point-like scatterer.

\section{Derivation of the moving point-scatterer model}
\label{sec:point-model-limit}

Point interactions describe spatially localized interactions in nuclear and
solid-state physics. In~\cite{AlbeverioHoeghKrohn1981}, point interactions
at fixed centers were obtained as limits of suitably scaled
short-range potentials. Following this approach, we derive the moving
point-scatterer model introduced in Section~\ref{sec:problem-formulation}
by concentrating a short-range potential around the trajectory $\boldsymbol y(t)$.
In the stationary case with a constant scattering parameter, the resulting
model is consistent with the established theory of point interactions and
fixed point scatterers~\cite{AlbeverioHoeghKrohn1981,NojaPosilicano2005,HuMantileSini2014}.
Our model extends this framework to moving scatterers.

We consider the wave equation
\begin{equation}\label{eq:short-range-model}
\begin{cases}
\displaystyle(c^{-2}\partial_t^2-\Delta)u_\varepsilon
-\frac{\lambda_\varepsilon(t)}{\varepsilon^2}
m\!\left(\frac{\boldsymbol x-\boldsymbol y(t)}{\varepsilon}\right)u_\varepsilon=f,
&\text{in }\mathbb R^3\times(0,T),\\[6pt]
u_\varepsilon=\partial_tu_\varepsilon=0,
&\text{in }\mathbb R^3\times(-\infty,0],
\end{cases}
\end{equation}
where
\[
f(\boldsymbol x,t)=
\begin{cases}
\varphi(t)\delta(\boldsymbol x-\boldsymbol z(t)),&t\ge0,\\[4pt]
0,&t<0.
\end{cases}
\]
The function $m$ is smooth, nonnegative, even, nonzero, and supported in a
ball $K$ centered at the origin. We assume that $\boldsymbol y$ has a strictly
subsonic $C^4$ extension to $[-1,T]$, and that the trajectories
$\boldsymbol y$ and $\boldsymbol z$ remain separated on $[0,T]$ as in
Section~\ref{sec:problem-formulation}. Choose $T$ sufficiently large that
$T>|\boldsymbol y(T)-\boldsymbol z(0)|/c$.
By separation and subsonicity, there is a unique $T_0\in(0,T)$ satisfying
\[
T_0=|\boldsymbol y(T_0)-\boldsymbol z(0)|/c.
\]
It is the first arrival time of the incident wave at the scatterer trajectory;
its existence and uniqueness are discussed in
Theorem~\ref{thm:arrival-time-maps}. We take $\alpha\in C([0,T])$, $\alpha\ge0$.

As $\varepsilon\to0$, the interaction region
$\boldsymbol y(t)+\varepsilon K$ shrinks to the scatterer position.
To retain a nontrivial scattering effect in this limit, we adjust the coupling
$\lambda_\varepsilon(t)$ together with the spatial scaling.
The following spectral data determine this choice and incorporate the
prescribed scattering parameter $\alpha(t)$ into the limiting response.

Put $\boldsymbol\beta=\boldsymbol y'/c$ and
$\gamma=(1-|\boldsymbol\beta|^2)^{-1/2}$, and define
\[
g_t(\boldsymbol r)=
\frac{1}{4\pi\sqrt{\gamma^{-2}|\boldsymbol r|^2+
(\boldsymbol\beta\cdot\boldsymbol r)^2}},\qquad
(\widetilde{\mathcal T}_t w)(\boldsymbol\xi)
=\int_K g_t(\boldsymbol\xi-\boldsymbol\eta)m(\boldsymbol\eta)
w(\boldsymbol\eta)\,d\boldsymbol\eta.
\]
Let $\nu_t$ be the largest eigenvalue of $\widetilde{\mathcal T}_t$, and
let $\psi_t$ be its eigenfunction normalized by
\begin{equation}\label{eq:short-range-tuning}
\nu_t^{-1}\int_K m(\boldsymbol\xi)\psi_t(\boldsymbol\xi)\,d\boldsymbol\xi=1.
\end{equation}
The quantities $\nu_t$ and $\psi_t$ are well defined, as will be justified
in detail in Lemma~\ref{lem:short-range-principal-eigenpair}.
We choose
\[
\lambda_\varepsilon(t)=\lambda_0(t)+\varepsilon\lambda_1(t),\qquad
\lambda_0(t)=\nu_t^{-1},\qquad
\lambda_1(t)=-\frac{\alpha(t)}{J_t},\qquad
J_t=\int_K m(\boldsymbol\xi)\psi_t(\boldsymbol\xi)^2\,d\boldsymbol\xi.
\]

\begin{lemma}\label{lem:short-range-principal-eigenpair}
	For every \(t\in[-1,T]\), \(\widetilde{\mathcal T}_t\) is compact and
	self-adjoint on \(L^2(K,m\,d\boldsymbol{\xi})\). Its largest eigenvalue
	\(\nu_t>0\) is simple, and the eigenfunction \(\psi_t\) normalized
	by~\eqref{eq:short-range-tuning} is positive and even on \(K\). Moreover,
	\[
	\nu, J, \lambda_0\in C^3([-1,T]),\qquad
	\psi \in C^3([-1,T];C(K)),
	\]
	where \(\nu(t):=\nu_t\), \(J(t):=J_t\), and \(\psi(t):=\psi_t\).
	In addition,
	\[
	0<\min_{\boldsymbol\xi\in K,\ t\in[-1,T]}\psi_t(\boldsymbol{\xi})
	\le\max_{\boldsymbol\xi\in K,\ t\in[-1,T]}\psi_t(\boldsymbol{\xi})<\infty.
	\]
\end{lemma}

\begin{proof}
Write $\mathcal H=L^2(K,m\,d\boldsymbol{\xi})$.
Since $\boldsymbol y\in C^4([-1,T];\mathbb R^3)$ is strictly subsonic, the definition of $g_t$ gives
\[
\sup_{t\in[-1,T],\,|\boldsymbol n|=1}
|\partial_t^k g_t(\boldsymbol n)|<\infty,\qquad 0\le k\le3.
\]
For $\boldsymbol r\ne0$, the definition of $g_t$ also gives
$g_t(\boldsymbol r)=|\boldsymbol r|^{-1}g_t(\boldsymbol r/|\boldsymbol r|)$.
Since $m$ is bounded and $|\boldsymbol r|^{-2}$ is locally integrable
in three dimensions,
\[
\sup_{\boldsymbol{\xi}\in K}
\int_K\frac{m(\boldsymbol{\eta})}{|\boldsymbol{\xi}-\boldsymbol{\eta}|^2}
\,d\boldsymbol{\eta}<\infty.
\]
Thus $g_t$ satisfies
\[
\int_K\int_K |g_t(\boldsymbol{\xi}-\boldsymbol{\eta})|^2
m(\boldsymbol{\xi})m(\boldsymbol{\eta})\,
d\boldsymbol{\eta}\,d\boldsymbol{\xi}<\infty.
\]
Hence $\widetilde{\mathcal T}_t$ is Hilbert--Schmidt on $\mathcal H$
and therefore compact.
The identity \(g_t(\boldsymbol r)=g_t(-\boldsymbol r)\) gives the self-adjointness of \(\widetilde{\mathcal T}_t\).

Fix \(t\in[-1,T]\). The largest eigenvalue \(\nu_t\) of the compact self-adjoint operator \(\widetilde{\mathcal T}_t\) satisfies
\[
\nu_t=\sup_{\|v\|_{\mathcal H}=1}
\langle\widetilde{\mathcal T}_t v,v\rangle_{\mathcal H}
\ge
\frac{\displaystyle\int_K\int_K
	g_t(\boldsymbol{\xi}-\boldsymbol{\eta})
	m(\boldsymbol{\xi})m(\boldsymbol{\eta})\,
	d\boldsymbol{\eta}\,d\boldsymbol{\xi}}
{\displaystyle\int_K m(\boldsymbol{\xi})\,d\boldsymbol{\xi}}>0.
\]
Let \(v\ne0\) be a real eigenfunction at \(\nu_t\). Since
\begin{align*}
	0&\ge
	\langle\widetilde{\mathcal T}_t|v|,|v|\rangle_{\mathcal H}
	-\langle\widetilde{\mathcal T}_t v,v\rangle_{\mathcal H}\\
	&=\int_K\int_K g_t(\boldsymbol{\xi}-\boldsymbol{\eta})
	\bigl[|v(\boldsymbol{\xi})||v(\boldsymbol{\eta})|
	-v(\boldsymbol{\xi})v(\boldsymbol{\eta})\bigr]\,
	m(\boldsymbol{\xi})m(\boldsymbol{\eta})\,
	d\boldsymbol{\eta}\,d\boldsymbol{\xi}\ge0,
\end{align*}
the two sets
\[
\{\boldsymbol{\xi}\in K: v(\boldsymbol{\xi})>0\}
\quad\text{and}\quad
\{\boldsymbol{\xi}\in K: v(\boldsymbol{\xi})<0\}
\]
cannot both have positive \(m\,d\boldsymbol{\xi}\) measure.
Thus, after changing its sign if necessary, \(v\ge0\) almost everywhere in \(\mathcal H\).
The eigenvalue equation therefore gives
\begin{equation}\label{eq:short-range-principal-eigenfunction-integral}
v(\boldsymbol{\xi})=
\frac1{\nu(t)}\int_K g_t(\boldsymbol{\xi}-\boldsymbol{\eta})
m(\boldsymbol{\eta})v(\boldsymbol{\eta})\,d\boldsymbol{\eta}>0,
\quad m\,d\boldsymbol\xi\text{-almost everywhere on }K.
\end{equation}
If the eigenspace at $\nu(t)$ had dimension at least two, it would
contain a nonzero real eigenfunction $w$ orthogonal to $v$.
The same argument allows us to choose $w>0$ almost everywhere, whereas
orthogonality requires
\[
0=\langle v,w\rangle_{\mathcal H}
=\int_K m(\boldsymbol{\xi})v(\boldsymbol{\xi})w(\boldsymbol{\xi})
\,d\boldsymbol{\xi}>0.
\]
This contradiction proves simplicity of $\nu(t)$.

Choose \(\chi \in C_c^\infty(\mathbb{R}^3)\) equal to \(1\) on the set \(K-K = \{\boldsymbol{\xi} - \boldsymbol{\eta} \mid \boldsymbol{\xi} \in K, \boldsymbol{\eta} \in K\}\). By the definition of \(g_t\) and the subsonicity of \(\boldsymbol{y}\), there exists a constant \(C\) such that
\[
\left|\frac{\partial^k}{\partial t^k} g_t(\boldsymbol{r})\right| \le \frac{C}{|\boldsymbol{r}|}, \quad (t\in[-1,T],\ \boldsymbol r\ne0,\ 0\le k\le3).
\]
By the dominated convergence theorem, the map
\[
t \mapsto \chi g_t
\]
belongs to \(C^3([-1, T]; L^2(\mathbb{R}^3))\). After extending \(m\) by zero to \(\mathbb{R}^3\), the map
\[
v \mapsto mv
\]
belongs to \(\mathcal{L}(\mathcal{H}; L^2(\mathbb{R}^3))\). 
By the choice of \(\chi\),
\[
\widetilde{\mathcal T}_t v=\bigl[(\chi g_t)*(mv)\bigr]\big|_K,\quad \forall v\in\mathcal H.
\]
Because convolution is a bounded bilinear map from
\(L^2(\mathbb R^3)\times L^2(\mathbb R^3)\) to \(C_0(\mathbb R^3)\)  and \(C(K) \subset \mathcal{H}\), it follows that
\[
t \mapsto\widetilde{\mathcal{T}}_t
\]
belongs to \(C^3([-1, T]; \mathcal{L}(\mathcal{H};C(K))) \subset C^3([-1, T]; \mathcal{L}(\mathcal{H}))\).

By the variational characterization of the largest eigenvalue of a
compact self-adjoint operator, for \(t_1,t_2\in[-1,T]\) we have
\[
|\nu_{t_1}-\nu_{t_2}|
\le\|\widetilde{\mathcal T}_{t_1}-\widetilde{\mathcal T}_{t_2}\|_{\mathcal L(\mathcal H)}.
\]
Since \(\nu_t\) is a simple isolated eigenvalue and
\(t\mapsto\widetilde{\mathcal T}_t\) belongs to
\(C^3([-1,T];\mathcal L(\mathcal H))\), the perturbation theorem for simple
isolated eigenvalues~\cite[p.~177]{Kloeckner2019}, together with the
preceding estimate, gives the following regularity by composition with
the local analytic eigenpair maps.
\begin{enumerate}
	\item The map \(t\mapsto\nu_t\) belongs to \(C^3([-1,T])\).
	\item The map \(t\mapsto\psi_t\) belongs to \(C^3([-1,T];C(K))\).
\end{enumerate}
Thus $J_t=\int_K m\psi_t^2$ is $C^3$ by the product rule.
Since $\nu_t>0$, $\lambda_0(t)=\nu_t^{-1}$ is $C^3$ as well.
The equation~\eqref{eq:short-range-principal-eigenfunction-integral} gives \(\psi_t(\boldsymbol{\xi})>0\)
for every \((\boldsymbol{\xi},t)\in K\times[-1,T]\). Since
\(t\mapsto\psi_t\) is continuous in \(C(K)\), the function
\((\boldsymbol{\xi},t)\mapsto\psi_t(\boldsymbol{\xi})\) is jointly
continuous on the compact set \(K\times[-1,T]\). Therefore,
\[
0<\min_{(\boldsymbol{\xi},t)\in K\times[-1,T]}\psi_t(\boldsymbol{\xi})
\le\max_{(\boldsymbol{\xi},t)\in K\times[-1,T]}\psi_t(\boldsymbol{\xi})
<\infty .
\]

Set \(\widetilde\psi_t(\boldsymbol{\xi})=\psi_t(-\boldsymbol{\xi})\).
By evenness of \(g_t\) and \(m\), we have
\begin{align*}
	(\widetilde{\mathcal T}_t\widetilde\psi_t)(-\boldsymbol{\xi})
	&=\int_K g_t(-\boldsymbol{\xi}-\boldsymbol{\eta})
	m(\boldsymbol{\eta})\psi_t(-\boldsymbol{\eta})\,d\boldsymbol{\eta}\\
	&=\int_K g_t(\boldsymbol{\xi}+\boldsymbol{\eta})
	m(\boldsymbol{\eta})\psi_t(-\boldsymbol{\eta})\,d\boldsymbol{\eta}\\
	&=\int_K g_t(\boldsymbol{\xi}-\boldsymbol{\eta})
	m(\boldsymbol{\eta})\psi_t(\boldsymbol{\eta})\,d\boldsymbol{\eta}\\
	&=\nu_t\psi_t(\boldsymbol{\xi})
	=\nu_t\widetilde\psi_t(-\boldsymbol{\xi}).
\end{align*}
Since \(\int_K m\widetilde\psi_t=\int_K m\psi_t=\nu_t\) and \(\nu_t\) is simple, it follows that \(\widetilde\psi_t(\boldsymbol{\xi})=\psi_t(\boldsymbol{\xi})\) in \(\mathcal H\). The equation~\eqref{eq:short-range-principal-eigenfunction-integral} then gives
\(\widetilde\psi_t=\psi_t\) pointwise on \(K\).
\end{proof}

We define
\begin{equation}\label{eq:short-range-normalized-operator}
\mathcal T_t=\lambda_0(t)\widetilde{\mathcal T}_t.
\end{equation}
Since $\widetilde{\mathcal T}_t\psi_t=\nu_t\psi_t$ and
$\lambda_0(t)=\nu_t^{-1}$, we have $\mathcal T_t\psi_t=\psi_t$.

Let $d_K=\operatorname{diam}K$. By the subsonic condition on \(\boldsymbol{y}\), for every
$\varepsilon\in[0,(c-c_0)/d_K)$, $(\boldsymbol{\xi},t)\in K\times[0,T]$,
and $\boldsymbol{\eta}\in K$, the equation
\begin{equation}\label{eq:short-range-retarded-time}
c(t-\theta)
-\bigl|\boldsymbol{y}(t)+\varepsilon\boldsymbol{\xi}
       -\boldsymbol{y}(\theta)-\varepsilon\boldsymbol{\eta}\bigr|
=0
\end{equation}
has a unique solution $\theta\in[-1,T]$.
Write $P_\varepsilon(\boldsymbol{\xi},t)=
\varepsilon u_\varepsilon(\boldsymbol{y}(t)+\varepsilon\boldsymbol{\xi},t)$
and $F_\varepsilon(\boldsymbol{\xi},t)=
u^{\mathrm{in}}(\boldsymbol{y}(t)+\varepsilon\boldsymbol{\xi},t)$,
where $u^{\mathrm{in}}$ is given by~\eqref{eq:incident-field}.
As in the derivation of the explicit expressions for $u^{\mathrm{in}}$ and
$u^{\mathrm{sc}}$ in Section~\ref{sec:problem-formulation}, convolution with
the free-space fundamental solution transforms
\eqref{eq:short-range-model} into
\begin{equation}\label{eq:short-range-inner}
	(I-\mathcal T_\varepsilon)P_\varepsilon=\varepsilon F_\varepsilon,
	\qquad
	(\mathcal T_\varepsilon w)(\boldsymbol{\xi},t)=
	\int_K\frac{\varepsilon\lambda_\varepsilon(\theta)m(\boldsymbol{\eta})}
	{4\pi R_\varepsilon J_\varepsilon}\,w(\boldsymbol{\eta},\theta)\,d\boldsymbol{\eta}
	\quad\text{in }K\times[0,T].
\end{equation}
Here 
\begin{equation}\label{eq:short-range-retarded-distance-jacobian}
R_\varepsilon=c(t-\theta), \quad J_\varepsilon=1-\frac{\boldsymbol{y}'(\theta)\cdot\left(\boldsymbol{y}(t)+\varepsilon\boldsymbol{\xi}
	-\boldsymbol{y}(\theta)-\varepsilon\boldsymbol{\eta}\right)}{cR_\varepsilon}.
\end{equation}

We next derive a first-order expansion of $\mathcal T_\varepsilon$ as
$\varepsilon\to0$ in Lemma~\ref{lem:short-range-operator-expansion}.
In Lemma~\ref{lem:short-range-inverse-estimate}, we establish the
invertibility of $I-\mathcal T_\varepsilon$ and estimate its inverse.

\begin{lemma}\label{lem:short-range-operator-expansion}
	For smooth \(\alpha\) and \(w\in C^2([-1,T];C(K))\), the operator \(\mathcal T_\varepsilon\) defined in
	\eqref{eq:short-range-inner} using the values of \(w\) on \([-1,T]\) satisfies
	\[
	\mathcal T_\varepsilon w=\mathcal T_t w+\varepsilon\mathcal S_t w
	+O(\varepsilon^2)
	\]
	uniformly on \(K\times[0,T]\), with \(\mathcal S_t\) being an operator whose definition is independent of \(\varepsilon\).
\end{lemma}
\begin{proof}

For $\boldsymbol\xi,\boldsymbol\eta\in K$ with $\boldsymbol\xi\ne\boldsymbol\eta$,
$t\in[0,T]$, and $\theta$ satisfying~\eqref{eq:short-range-retarded-time}, set
\begin{gather*}
	\ell=|\boldsymbol\xi-\boldsymbol\eta|,\quad
	\boldsymbol n=\frac{\boldsymbol\xi-\boldsymbol\eta}{|\boldsymbol\xi-\boldsymbol\eta|},\quad
	\boldsymbol\beta=\frac{\boldsymbol y'(t)}{c},\quad
	\boldsymbol a=\boldsymbol y''(t),\\
	b=\boldsymbol\beta\cdot\boldsymbol n
	=\frac{\boldsymbol y'(t)}{c}\cdot
	\frac{\boldsymbol\xi-\boldsymbol\eta}{|\boldsymbol\xi-\boldsymbol\eta|},\quad
	\gamma=(1-|\boldsymbol\beta|^2)^{-1/2}
	=\left(1-\frac{|\boldsymbol y'(t)|^2}{c^2}\right)^{-1/2},\\
	D=(\gamma^{-2}+b^2)^{1/2}
	=\left[1-\frac{|\boldsymbol y'(t)|^2}{c^2}
	+\left(\frac{\boldsymbol y'(t)}{c}\cdot
	\frac{\boldsymbol\xi-\boldsymbol\eta}{|\boldsymbol\xi-\boldsymbol\eta|}\right)^2\right]^{1/2},\\
	\kappa=\gamma^2(D+b)
	=\frac{\displaystyle
		\left[1-\frac{|\boldsymbol y'(t)|^2}{c^2}
		+\left(\frac{\boldsymbol y'(t)}{c}\cdot
		\frac{\boldsymbol\xi-\boldsymbol\eta}{|\boldsymbol\xi-\boldsymbol\eta|}\right)^2\right]^{1/2}
		+\frac{\boldsymbol y'(t)}{c}\cdot
		\frac{\boldsymbol\xi-\boldsymbol\eta}{|\boldsymbol\xi-\boldsymbol\eta|}}
	{\displaystyle 1-\frac{|\boldsymbol y'(t)|^2}{c^2}},
\end{gather*}
\begin{equation}\label{eq:short-range-retarded-displacement-delay}
	\boldsymbol r_\varepsilon=\boldsymbol y(t)+\varepsilon\boldsymbol\xi
	-\boldsymbol y(\theta)-\varepsilon\boldsymbol\eta,\quad
	\delta=t-\theta>0.
\end{equation}
Of these quantities, only $\boldsymbol r_\varepsilon$ and \(\delta\) depend on \(\varepsilon\).

Set
\begin{equation}\label{eq:short-range-geometric-correction-coefficients}
G_t(\boldsymbol n)=
\frac{\boldsymbol a\cdot\boldsymbol n+
\kappa^2(2D-b)\boldsymbol\beta\cdot\boldsymbol a}{8\pi c^2D^3},
\qquad
H_t(\boldsymbol n)=\frac{\kappa}{4\pi cD},
\end{equation}
and define $\mathcal S_t$ by
\begin{equation}\label{eq:short-range-first-order-operator}
\begin{aligned}
(\mathcal S_t w)(\boldsymbol\xi,t)
=\int_K m(\boldsymbol\eta)\bigl[
&\lambda_1(t)g_t(\boldsymbol\xi-\boldsymbol\eta)w(\boldsymbol\eta,t)\\
&-H_t(\boldsymbol n)\partial_t(\lambda_0w)(\boldsymbol\eta,t)\\
&-\lambda_0(t)G_t(\boldsymbol n)w(\boldsymbol\eta,t)
\bigr]\,d\boldsymbol\eta.
\end{aligned}
\end{equation}

By \eqref{eq:short-range-retarded-time} and
\eqref{eq:short-range-retarded-displacement-delay}, we have
\[
c\delta=|\varepsilon\ell\boldsymbol n+\boldsymbol y(t)-\boldsymbol y(t-\delta)|.
\]
By the regularity and subsonic speed bound of $\boldsymbol y$,
\[
\varepsilon\ell-c_0\delta
\le |\varepsilon\ell\boldsymbol n|-|\boldsymbol y(t)-\boldsymbol y(t-\delta)|
\le c\delta
\le |\varepsilon\ell\boldsymbol n|+|\boldsymbol y(t)-\boldsymbol y(t-\delta)|
\le\varepsilon\ell+c_0\delta.
\]
Hence
\begin{equation}\label{eq:short-range-delay-bounds}
\frac{\varepsilon\ell}{c+c_0}\le\delta
\le\frac{\varepsilon\ell}{c-c_0}.
\end{equation}
Thus $\delta\to0$ as $\varepsilon\to0^+$, and below we regard
$\delta$ and $\varepsilon\ell$ as small quantities of the same order.
Taylor expansion gives
\begin{align}
\boldsymbol r_\varepsilon
&=\varepsilon\ell\boldsymbol n+\boldsymbol y(t)-\boldsymbol y(t-\delta)\notag\\
&=\varepsilon\ell\boldsymbol n+\delta\boldsymbol y'(t)
-\tfrac12\delta^2\boldsymbol y''(t)+O(\delta^3)\notag\\
&=\varepsilon\ell\boldsymbol n+c\delta\boldsymbol\beta
-\tfrac12\delta^2\boldsymbol a+O(\varepsilon^3\ell^3).
\label{eq:short-range-retarded-displacement-expansion}
\end{align}
Consequently, $c\delta=|\boldsymbol r_\varepsilon|$ yields
\[
c^2\delta^2=\left|\varepsilon\ell\boldsymbol n+c\delta\boldsymbol\beta
-\tfrac12\delta^2\boldsymbol a+O(\varepsilon^3\ell^3)\right|^2.
\]
Expanding and rearranging, we obtain
\begin{equation}\label{eq:short-range-delay-factorization}
\gamma^{-2}c^2
\left(\delta-\frac{\varepsilon\ell\kappa}{c}\right)
\left(\delta+\frac{\varepsilon\ell}{c(D+b)}\right)
=-\delta^2(\varepsilon\ell\boldsymbol n+c\delta\boldsymbol\beta)
\cdot\boldsymbol a+O(\varepsilon^4\ell^4).
\end{equation}
Since
\[
-\delta^2(\varepsilon\ell\boldsymbol n+c\delta\boldsymbol\beta)
\cdot\boldsymbol a+O(\varepsilon^4\ell^4)=O(\varepsilon^3\ell^3)
\]
and
\[
\delta+\frac{\varepsilon\ell}{c(D+b)}
\ge\frac{\varepsilon\ell}{c+c_0}
+\frac{\varepsilon\ell}{\sqrt{c^2+c_0^2}+c_0},
\]
we have $\delta-\varepsilon\ell\kappa/c=O(\varepsilon^2\ell^2)$, and hence
\begin{equation}\label{eq:short-range-delay-leading-expansion}
\delta=\frac{\varepsilon\ell\kappa}{c}+O(\varepsilon^2\ell^2),
\qquad
\delta^2=\frac{\varepsilon^2\ell^2\kappa^2}{c^2}
+O(\varepsilon^3\ell^3).
\end{equation}
It follows that
\begin{equation}\label{eq:short-range-delay-cubic-term}
\begin{aligned}
&-\delta^2(\varepsilon\ell\boldsymbol n+c\delta\boldsymbol\beta)
\cdot\boldsymbol a+O(\varepsilon^4\ell^4)\\
&\qquad=-\frac{\varepsilon^3\ell^3\kappa^2}{c^2}
(\boldsymbol n+\kappa\boldsymbol\beta)\cdot\boldsymbol a
+O(\varepsilon^4\ell^4).
\end{aligned}
\end{equation}
Since
\[
\left.\frac{\partial}{\partial\delta}
\left(\gamma^{-2}c^2\delta^2-2bc\varepsilon\ell\delta
-\varepsilon^2\ell^2\right)\right|_{\delta=\varepsilon\ell\kappa/c}
=2c\varepsilon\ell D,
\]
we obtain
\begin{equation}\label{eq:short-range-delay-quadratic-linearization}
\gamma^{-2}c^2\delta^2-2bc\varepsilon\ell\delta-\varepsilon^2\ell^2
=2c\varepsilon\ell D
\left(\delta-\frac{\varepsilon\ell\kappa}{c}\right)
+O(\varepsilon^4\ell^4).
\end{equation}
Substituting~\eqref{eq:short-range-delay-cubic-term} and
\eqref{eq:short-range-delay-quadratic-linearization} into
\eqref{eq:short-range-delay-factorization} and rearranging gives
\[
\delta=\frac{\varepsilon\ell\kappa}{c}
-\frac{\varepsilon^2\ell^2\kappa^2
(\boldsymbol n+\kappa\boldsymbol\beta)\cdot\boldsymbol a}{2c^3D}
+O(\varepsilon^3\ell^3).
\]

Since $\boldsymbol y'(t-\delta)=c\boldsymbol\beta-\delta\boldsymbol a
+O(\varepsilon^2\ell^2)$, combining
\eqref{eq:short-range-retarded-distance-jacobian},
\eqref{eq:short-range-retarded-displacement-expansion}, and
\eqref{eq:short-range-delay-leading-expansion}, we obtain
\begin{align*}
R_\varepsilon J_\varepsilon
&=c\delta\left(1-
\frac{\boldsymbol y'(t-\delta)\cdot\boldsymbol r_\varepsilon}{c^2\delta}\right)\\
&=\gamma^{-2}c\delta-\varepsilon\ell b
+\frac{\varepsilon\ell\delta}{c}\boldsymbol a\cdot\boldsymbol n
+\tfrac32\delta^2\boldsymbol\beta\cdot\boldsymbol a
+O(\varepsilon^3\ell^3)\\
&=\varepsilon\ell D\left[1+\frac{\varepsilon\ell}{2c^2D^2}
\left(\boldsymbol a\cdot\boldsymbol n
+\kappa^2(2D-b)\boldsymbol\beta\cdot\boldsymbol a\right)\right]
+O(\varepsilon^3\ell^3).
\end{align*}
Therefore,
\begin{align}
\frac{\varepsilon}{4\pi R_\varepsilon J_\varepsilon}
&=\frac{1}{4\pi\ell D}\left[1-\frac{\varepsilon\ell}{2c^2D^2}
\left(\boldsymbol a\cdot\boldsymbol n
+\kappa^2(2D-b)\boldsymbol\beta\cdot\boldsymbol a\right)
+O(\varepsilon^2\ell^2)\right]\notag\\
&=g_t(\boldsymbol\xi-\boldsymbol\eta)-\varepsilon G_t(\boldsymbol n)
+O(\varepsilon^2\ell).
\label{eq:short-range-geometric-factor-expansion}
\end{align}

Note that
\begin{align*}
(\lambda_0w)(\boldsymbol\eta,t-\delta)
&=(\lambda_0w)(\boldsymbol\eta,t)
-\delta\,\partial_t(\lambda_0w)(\boldsymbol\eta,t)
+\int_0^\delta(\delta-s)\partial_t^2(\lambda_0w)(\boldsymbol\eta,t-s)\,ds\\
&=(\lambda_0w)(\boldsymbol\eta,t)
-\delta\,\partial_t(\lambda_0w)(\boldsymbol\eta,t)+O(\delta^2)\\
&=(\lambda_0w)(\boldsymbol\eta,t)
-\delta\,\partial_t(\lambda_0w)(\boldsymbol\eta,t)+O(\varepsilon^2\ell^2),
\end{align*}
and
\begin{align*}
(\lambda_1w)(\boldsymbol\eta,t-\delta)
&=(\lambda_1w)(\boldsymbol\eta,t)
-\int_0^\delta\partial_t(\lambda_1w)(\boldsymbol\eta,t-s)\,ds\\
&=(\lambda_1w)(\boldsymbol\eta,t)+O(\delta)\\
&=(\lambda_1w)(\boldsymbol\eta,t)+O(\varepsilon\ell).
\end{align*}
Consequently,
\begin{align}
\lambda_\varepsilon(\theta)w(\boldsymbol\eta,\theta)
&=(\lambda_0w)(\boldsymbol\eta,t-\delta)
+\varepsilon(\lambda_1w)(\boldsymbol\eta,t-\delta)\notag\\
&=(\lambda_0w)(\boldsymbol\eta,t)
+\varepsilon(\lambda_1w)(\boldsymbol\eta,t)
-\delta\,\partial_t(\lambda_0w)(\boldsymbol\eta,t)
+O\bigl(\varepsilon^2(\ell+\ell^2)\bigr)\notag\\
&=(\lambda_0w)(\boldsymbol\eta,t)
+\varepsilon\left[(\lambda_1w)(\boldsymbol\eta,t)
-\frac{\ell\kappa}{c}\partial_t(\lambda_0w)(\boldsymbol\eta,t)\right]
+O\bigl(\varepsilon^2(\ell+\ell^2)\bigr),
\label{eq:short-range-retarded-function-factor-expansion}
\end{align}
where the last equality follows from
\eqref{eq:short-range-delay-leading-expansion}.

Since $\lambda_\varepsilon(\theta)=\lambda_0(\theta)
+\varepsilon\lambda_1(\theta)$, there exists a constant $C_1>0$ such that
\begin{equation}\label{eq:short-range-function-factor-bound}
|\lambda_\varepsilon(\theta)w(\boldsymbol\eta,\theta)|\le C_1.
\end{equation}
Since
\[
g_t(\boldsymbol r)=\frac{1}{4\pi
\sqrt{\gamma^{-2}|\boldsymbol r|^2+(\boldsymbol\beta\cdot\boldsymbol r)^2}},
\]
there exists a constant $C_2>0$ such that
\begin{equation}\label{eq:short-range-truncated-geometric-factor-bound}
|g_t(\boldsymbol\xi-\boldsymbol\eta)-\varepsilon G_t(\boldsymbol n)|
\le C_2\left(\frac1\ell+1\right).
\end{equation}
Similarly, there exists a constant $C_3>0$ such that
\begin{equation}\label{eq:short-range-first-order-product-bound}
|G_t(\boldsymbol n)|
\left|(\lambda_1w)(\boldsymbol\eta,t)
-\frac{\ell\kappa}{c}\partial_t(\lambda_0w)(\boldsymbol\eta,t)\right|
\le C_3(1+\ell).
\end{equation}
The constants $C_1,C_2,C_3$ depend only on
$\boldsymbol y$, $m$, $\alpha$ and $w$, and are independent of
$\boldsymbol\xi$, $\boldsymbol\eta$, $t$ and sufficiently small $\varepsilon$.
In fact, there also exist constants $C_4,C_5>0$ with the same dependence
such that~\eqref{eq:short-range-geometric-factor-expansion} and
\eqref{eq:short-range-retarded-function-factor-expansion} give
\begin{equation}\label{eq:short-range-geometric-factor-remainder-bound}
\left|\frac{\varepsilon}{4\pi R_\varepsilon J_\varepsilon}
-g_t(\boldsymbol\xi-\boldsymbol\eta)+\varepsilon G_t(\boldsymbol n)\right|
\le C_4\varepsilon^2\ell
\end{equation}
and
\begin{equation}\label{eq:short-range-function-factor-remainder-bound}
\begin{aligned}
\Bigl|\lambda_\varepsilon(\theta)w(\boldsymbol\eta,\theta)
&-(\lambda_0w)(\boldsymbol\eta,t)-\varepsilon(\lambda_1w)(\boldsymbol\eta,t)\\
&+\frac{\varepsilon\ell\kappa}{c}\partial_t(\lambda_0w)(\boldsymbol\eta,t)\Bigr|
\le C_5\varepsilon^2(\ell+\ell^2).
\end{aligned}
\end{equation}
Combining~\eqref{eq:short-range-function-factor-bound},
\eqref{eq:short-range-truncated-geometric-factor-bound},
\eqref{eq:short-range-first-order-product-bound},
\eqref{eq:short-range-geometric-factor-remainder-bound}, and
\eqref{eq:short-range-function-factor-remainder-bound}, we obtain
\begin{align*}
&\Biggl|\frac{\varepsilon\lambda_\varepsilon(\theta)w(\boldsymbol\eta,\theta)}
{4\pi R_\varepsilon J_\varepsilon}
-g_t(\boldsymbol\xi-\boldsymbol\eta)(\lambda_0w)(\boldsymbol\eta,t)\\
&\qquad-\varepsilon\Bigl[
g_t(\boldsymbol\xi-\boldsymbol\eta)(\lambda_1w)(\boldsymbol\eta,t)
-H_t(\boldsymbol n)\partial_t(\lambda_0w)(\boldsymbol\eta,t)
-G_t(\boldsymbol n)(\lambda_0w)(\boldsymbol\eta,t)\Bigr]\Biggr|\\
&=\Biggl|\left[\frac{\varepsilon}{4\pi R_\varepsilon J_\varepsilon}
-g_t(\boldsymbol\xi-\boldsymbol\eta)+\varepsilon G_t(\boldsymbol n)\right]
\lambda_\varepsilon(\theta)w(\boldsymbol\eta,\theta)\\
&\qquad+\left[g_t(\boldsymbol\xi-\boldsymbol\eta)-\varepsilon G_t(\boldsymbol n)\right]
\Bigl[\lambda_\varepsilon(\theta)w(\boldsymbol\eta,\theta)
-(\lambda_0w)(\boldsymbol\eta,t)\\
&\hspace{45mm}-\varepsilon(\lambda_1w)(\boldsymbol\eta,t)
+\frac{\varepsilon\ell\kappa}{c}\partial_t(\lambda_0w)(\boldsymbol\eta,t)\Bigr]\\
&\qquad-\varepsilon^2G_t(\boldsymbol n)
\left[(\lambda_1w)(\boldsymbol\eta,t)
-\frac{\ell\kappa}{c}\partial_t(\lambda_0w)(\boldsymbol\eta,t)\right]\Biggr|\\
&\le C_1C_4\varepsilon^2\ell
+C_2C_5\varepsilon^2(1+\ell)^2
+C_3\varepsilon^2(1+\ell)
\le C\varepsilon^2(1+d_K).
\end{align*}
Therefore,
\[
\left|(\mathcal T_\varepsilon w)(\boldsymbol\xi,t)
-(\mathcal T_t w)(\boldsymbol\xi,t)
-\varepsilon(\mathcal S_t w)(\boldsymbol\xi,t)\right|
\le C\varepsilon^2(1+d_K)\int_K m(\boldsymbol\eta)\,d\boldsymbol\eta,
\]
which holds for all \((\boldsymbol\xi,t)\in K\times[0,T]\).
\end{proof}

\begin{lemma}\label{lem:short-range-inverse-estimate}
For bounded nonnegative $\alpha$, the operator $I-\mathcal T_\varepsilon$
is invertible on bounded functions on $K\times[0,T]$ extended by zero
to $t<0$, and there is a constant $C>0$, independent of sufficiently
small $\varepsilon$, such that
\begin{equation}\label{eq:short-range-inverse-supremum-bound}
\bigl\|(I-\mathcal T_\varepsilon)^{-1}\bigr\|_{\infty\to\infty}
\le \frac{C}{\varepsilon}.
\end{equation}
If $F\ge0$ on $K\times[0,T]$, then
$(I-\mathcal T_\varepsilon)^{-1}F\ge0$ on $K\times[0,T]$.
\end{lemma}
\begin{proof}

Let $M>0$ be chosen later, and set
\[
W(\boldsymbol\xi,t):=e^{Mt}\psi_t(\boldsymbol\xi),
\qquad (\boldsymbol\xi,t)\in K\times[-1,T],
\]
and
\[
W_c(\boldsymbol\xi,t):=
\begin{cases}
e^{Mt}\psi_t(\boldsymbol\xi), & (\boldsymbol\xi,t)\in K\times[0,T],\\
0, & (\boldsymbol\xi,t)\in K\times(-\infty,0).
\end{cases}
\]
Write
\begin{align*}
(\mathcal T_\varepsilon^{[0]}W_c)(\boldsymbol\xi,t)
&:=\int_{\{\boldsymbol\eta\in K:\,\theta\ge0\}}
\frac{\varepsilon\lambda_0(\theta)m(\boldsymbol\eta)}
{4\pi R_\varepsilon J_\varepsilon}
e^{M\theta}\psi_\theta(\boldsymbol\eta)\,d\boldsymbol\eta,\\
(\widehat{\mathcal T}_\varepsilon^{[0]}W)(\boldsymbol\xi,t)
&:=\int_K
\frac{\varepsilon\lambda_0(\theta)m(\boldsymbol\eta)}
{4\pi R_\varepsilon J_\varepsilon}
e^{M\theta}\psi_\theta(\boldsymbol\eta)\,d\boldsymbol\eta.
\end{align*}
Applying Lemma~\ref{lem:short-range-operator-expansion} with $\alpha=0$ gives
\begin{equation}\label{eq:short-range-auxiliary-function-expansion}
\begin{aligned}
(\widehat{\mathcal T}_\varepsilon^{[0]}W)(\boldsymbol\xi,t)
={}&(\mathcal T_tW)(\boldsymbol\xi,t)\\
&-\varepsilon e^{Mt}\int_K m(\boldsymbol\eta)
\bigl[H_t(\boldsymbol n)\partial_t(\lambda_0\psi_t)(\boldsymbol\eta)
+\lambda_0(t)G_t(\boldsymbol n)\psi_t(\boldsymbol\eta)\bigr]
\,d\boldsymbol\eta\\
&-\varepsilon e^{Mt}M\lambda_0(t)\int_K
H_t(\boldsymbol n)m(\boldsymbol\eta)\psi_t(\boldsymbol\eta)
\,d\boldsymbol\eta
+\varepsilon^2\Theta(\boldsymbol\xi,t;M),
\end{aligned}
\end{equation}
where, for each fixed $M$, $\Theta(\boldsymbol\xi,t;M)$ is bounded on
$K\times[0,T]$.

By Lemma~\ref{lem:short-range-principal-eigenpair}, $\lambda_0\psi_t$ and
$\partial_t(\lambda_0\psi_t)$ are bounded on $K\times[-1,T]$.
By Lemma~\ref{lem:short-range-operator-expansion}, $H_t$ and $G_t$ are
bounded on $\mathbb S^2\times[-1,T]$. Hence
\[
\left|\int_K m(\boldsymbol\eta)
\bigl[H_t(\boldsymbol n)\partial_t(\lambda_0\psi_t)(\boldsymbol\eta)
+\lambda_0(t)G_t(\boldsymbol n)\psi_t(\boldsymbol\eta)\bigr]
\,d\boldsymbol\eta\right|
\]
is uniformly bounded for $(\boldsymbol\xi,t)\in K\times[0,T]$.
Choose $M$ such that
\begin{equation}\label{eq:short-range-exponential-rate-choice}
\frac{M}{8\pi c}>
\sup_{(\boldsymbol\xi,t)\in K\times[0,T]}
\left|\int_K m(\boldsymbol\eta)
\bigl[H_t(\boldsymbol n)\partial_t(\lambda_0\psi_t)(\boldsymbol\eta)
+\lambda_0(t)G_t(\boldsymbol n)\psi_t(\boldsymbol\eta)\bigr]
\,d\boldsymbol\eta\right|+1.
\end{equation}
By~\eqref{eq:short-range-geometric-correction-coefficients}, we have $H_t\ge1/(8\pi c)$ on
$\mathbb S^2\times[-1,T]$. Combining this bound with
\eqref{eq:short-range-tuning} gives
\begin{equation}\label{eq:short-range-delay-damping-lower-bound}
\begin{aligned}
\varepsilon e^{Mt}M\lambda_0(t)\int_K
H_t(\boldsymbol n)m(\boldsymbol\eta)\psi_t(\boldsymbol\eta)\,d\boldsymbol\eta
&\ge\varepsilon e^{Mt}\frac{M}{8\pi c}\lambda_0(t)
\int_K m(\boldsymbol\eta)\psi_t(\boldsymbol\eta)\,d\boldsymbol\eta\\
&=\varepsilon e^{Mt}\frac{M}{8\pi c}.
\end{aligned}
\end{equation}
For this fixed $M$, choose $\varepsilon$ sufficiently small that
\begin{equation}\label{eq:short-range-auxiliary-remainder-bound}
|\varepsilon^2\Theta(\boldsymbol\xi,t;M)|
\le\frac{\varepsilon e^{Mt}}2
\qquad\text{on }K\times[0,T].
\end{equation}
Using~\eqref{eq:short-range-auxiliary-function-expansion},
\eqref{eq:short-range-exponential-rate-choice},
\eqref{eq:short-range-delay-damping-lower-bound}, and
\eqref{eq:short-range-auxiliary-remainder-bound},
and using $\mathcal T_tW=W$, we obtain
\begin{align*}
(\widehat{\mathcal T}_\varepsilon^{[0]}W)(\boldsymbol\xi,t)
\le{}&(\mathcal T_tW)(\boldsymbol\xi,t)+\frac{\varepsilon e^{Mt}}2\\
&-\varepsilon e^{Mt}\left(
\int_K m(\boldsymbol\eta)
\bigl[H_t(\boldsymbol n)\partial_t(\lambda_0\psi_t)(\boldsymbol\eta)
+\lambda_0(t)G_t(\boldsymbol n)\psi_t(\boldsymbol\eta)\bigr]
\,d\boldsymbol\eta+\frac{M}{8\pi c}\right)
\\
\le{}&(\mathcal T_tW)(\boldsymbol\xi,t)-\varepsilon e^{Mt}
+\frac{\varepsilon e^{Mt}}2\\
={}&W(\boldsymbol\xi,t)-\frac{\varepsilon e^{Mt}}2\\
={}&W(\boldsymbol\xi,t)-\frac{\varepsilon}{2}
\frac{W(\boldsymbol\xi,t)}{\psi_t(\boldsymbol\xi)}\\
\le{}&(1-C_*\varepsilon)W(\boldsymbol\xi,t),
\end{align*}
where
\[
C_*:=\frac{1}{2\sup_{(\boldsymbol\xi,t)\in K\times[0,T]}
\psi_t(\boldsymbol\xi)}.
\]
Since $\alpha$ is nonnegative and bounded and $\lambda_0,J_t$ have
positive lower bounds, we may further decrease $\varepsilon$ so that
\[
\varepsilon\sup_{t\in[0,T]}\frac{\alpha(t)}{\lambda_0(t)J_t}\le1.
\]
Consequently,
\[
0\le\lambda_\varepsilon(t)=\lambda_0(t)
-\varepsilon\frac{\alpha(t)}{J_t}\le\lambda_0(t),
\qquad t\in[0,T].
\]
The nonnegative kernels and the definitions of $W$ and $W_c$ therefore give
\[
0\le(\mathcal T_\varepsilon W_c)(\boldsymbol\xi,t)
\le(\mathcal T_\varepsilon^{[0]}W_c)(\boldsymbol\xi,t)
\le(\widehat{\mathcal T}_\varepsilon^{[0]}W)(\boldsymbol\xi,t)
\le(1-C_*\varepsilon)W(\boldsymbol\xi,t)
\]
for all $(\boldsymbol\xi,t)\in K\times[0,T]$.
By Lemma~\ref{lem:short-range-principal-eigenpair}, $\psi_t$ has a positive
lower bound and a finite upper bound on $K\times[0,T]$. Thus we define
\[
W_{\min}:=\inf_{(\boldsymbol\xi,t)\in K\times[0,T]}W(\boldsymbol\xi,t)>0,
\qquad
W_{\max}:=\sup_{(\boldsymbol\xi,t)\in K\times[0,T]}W(\boldsymbol\xi,t)<\infty.
\]
For every bounded function $w$ vanishing for $t<0$, we have
\[
|w|\le\frac{\|w\|_\infty}{W_{\min}}W_c.
\]
By positivity of the kernel of $\mathcal T_\varepsilon$, induction gives
\[
|\mathcal T_\varepsilon^k w|
\le\frac{\|w\|_\infty}{W_{\min}}\mathcal T_\varepsilon^kW_c
\le\frac{\|w\|_\infty}{W_{\min}}(1-C_*\varepsilon)^kW
\le\|w\|_\infty(1-C_*\varepsilon)^k\frac{W_{\max}}{W_{\min}}
\]
on $K\times[0,T]$. Hence
\[
\|\mathcal T_\varepsilon^k\|_{\infty\to\infty}
\le\frac{W_{\max}}{W_{\min}}(1-C_*\varepsilon)^k.
\]
For $0<C_*\varepsilon<1$, the Neumann series therefore converges in
operator norm and gives
\begin{equation}\label{eq:short-range-inverse-neumann-series}
(I-\mathcal T_\varepsilon)^{-1}=\sum_{k=0}^{\infty}\mathcal T_\varepsilon^k.
\end{equation}
It follows that
\[
\|(I-\mathcal T_\varepsilon)^{-1}\|_{\infty\to\infty}
\le\frac{W_{\max}}{C_*W_{\min}}\frac1\varepsilon.
\]
Since each $\mathcal T_\varepsilon^k$ maps nonnegative functions to
nonnegative functions,~\eqref{eq:short-range-inverse-neumann-series}
also gives $(I-\mathcal T_\varepsilon)^{-1}F\ge0$ for $F\ge0$.

\end{proof}

For each fixed $t$, define the linear functional $\ell_t$ on $C(K)$ by
\[
\ell_t(v)=\int_K m(\boldsymbol{\xi})\psi_t(\boldsymbol{\xi})
v(\boldsymbol{\xi})\,d\boldsymbol{\xi},\qquad v\in C(K).
\]
By Lemma~\ref{lem:short-range-operator-expansion}, for
$w\in C^2([-1,T];C(K))$, 
$\mathcal T_t w$ is the leading-order term in the expansion of
$\mathcal T_\varepsilon w$.
Thus, to construct an approximation to $P_\varepsilon$ in
\eqref{eq:short-range-inner}, we will need to solve equations involving
$I-\mathcal T_t$. Since $\mathcal T_t\psi_t=\psi_t$, the operator $I-\mathcal T_t$
is not invertible. The following lemma identifies the compatibility
condition for solvability and gives an inverse on $\ker\ell_t$
that depends regularly on $t$.

\begin{lemma}\label{lem:short-range-frozen-solvability}
For each $t\in[0,T]$, the following statements hold.
\begin{enumerate}
\item For every $v\in C(K)$,
\begin{equation}\label{eq:short-range-annihilation}
\ell_t((I-\mathcal T_t)v)=0.
\end{equation}
\item The restriction
\[
\left.(I-\mathcal T_t)\right|_{\ker\ell_t}\colon
\ker\ell_t\longrightarrow\ker\ell_t
\]
is a bijection.
\item Define the projection $\Pi_t\colon C(K)\to C(K)$ by
\begin{equation}\label{eq:short-range-resonant-projection}
\Pi_t v=\psi_t\frac{\ell_t(v)}{\ell_t(\psi_t)}
       =\psi_t\frac{\ell_t(v)}{J_t},\qquad v\in C(K).
\end{equation}
Then $I-\mathcal T_t+\Pi_t$ is invertible on $C(K)$. Define
$\mathcal R\colon [0,T]\to\mathcal L(C(K),C(K))$ by
\[
\mathcal R(t)=(I-\mathcal T_t+\Pi_t)^{-1}.
\]
The map $\mathcal R$ is $C^3$ in the operator norm.
\end{enumerate}
\end{lemma}
\begin{proof}
The integrable bound $C/|\boldsymbol{\xi}-\boldsymbol{\eta}|$ justifies
exchanging the integrals. Using
$g_t(\boldsymbol{\xi}-\boldsymbol{\eta})=g_t(\boldsymbol{\eta}-\boldsymbol{\xi})$
and $\mathcal T_t\psi_t=\psi_t$, for any $v\in C(K)$ we obtain
\begin{align*}
&\quad\ell_t((I-\mathcal T_t)v)\\
&=\int_K m(\boldsymbol{\xi})\psi_t(\boldsymbol{\xi})
   ((I-\mathcal T_t)v)(\boldsymbol{\xi})\,d\boldsymbol{\xi}\\
&=\int_K m(\boldsymbol{\xi})\psi_t(\boldsymbol{\xi})v(\boldsymbol{\xi})\,d\boldsymbol{\xi}
-\int_K m(\boldsymbol{\xi})\psi_t(\boldsymbol{\xi})\lambda_0(t)
   \int_K g_t(\boldsymbol{\xi}-\boldsymbol{\eta})m(\boldsymbol{\eta})
   v(\boldsymbol{\eta})\,d\boldsymbol{\eta}\,d\boldsymbol{\xi}\\
&=\int_K m(\boldsymbol{\xi})\psi_t(\boldsymbol{\xi})v(\boldsymbol{\xi})\,d\boldsymbol{\xi}
-\int_K m(\boldsymbol{\eta})v(\boldsymbol{\eta})\lambda_0(t)
   \int_K g_t(\boldsymbol{\eta}-\boldsymbol{\xi})m(\boldsymbol{\xi})
   \psi_t(\boldsymbol{\xi})\,d\boldsymbol{\xi}\,d\boldsymbol{\eta}\\
&=\int_K m(\boldsymbol{\xi})
   \bigl[\psi_t(\boldsymbol{\xi})-(\mathcal T_t\psi_t)(\boldsymbol{\xi})\bigr]
   v(\boldsymbol{\xi})\,d\boldsymbol{\xi}\\
&=0.
\end{align*}

By the definition of $\widetilde{\mathcal T}_t$ and
\eqref{eq:short-range-normalized-operator}, for $w\in C(K)$ we have
\[
(\mathcal T_t w)(\boldsymbol{\xi})
=\int_K \lambda_0(t)g_t(\boldsymbol{\xi}-\boldsymbol{\eta})
 m(\boldsymbol{\eta})w(\boldsymbol{\eta})\,d\boldsymbol{\eta}.
\]
Its kernel \( \lambda_0(t)g_t(\boldsymbol{\xi}-\boldsymbol{\eta})
m(\boldsymbol{\eta})\) is continuous on
$\{(\boldsymbol{\xi},\boldsymbol{\eta})\in K^2:
\boldsymbol{\xi}\ne\boldsymbol{\eta}\}$ and satisfies
\[
\bigl|\lambda_0(t)g_t(\boldsymbol{\xi}-\boldsymbol{\eta})
m(\boldsymbol{\eta})\bigr|
\le \frac{C}{|\boldsymbol{\xi}-\boldsymbol{\eta}|}
\qquad (\boldsymbol{\xi}\ne\boldsymbol{\eta}).
\]
This bound makes $\mathcal T_t$ a weakly singular
integral operator on $C(K)$. It is therefore compact by
\cite[Theorem~2.22]{Kress1999}.
Let $v\in C(K)$ satisfy $\mathcal T_t v=v$. Since $1$ is a simple
eigenvalue of $\mathcal T_t$ in $L^2(K,m\,d\boldsymbol{\xi})$, there is a
constant $C_v$ such that
\[
v=C_v\psi_t\qquad m\,d\boldsymbol{\xi}\text{-almost everywhere on }K.
\]
By the definition of $\widetilde{\mathcal T}_t$ and~\eqref{eq:short-range-normalized-operator}, for every \(\boldsymbol{\xi}\in K\) we obtain
\[
v(\boldsymbol{\xi})-C_v\psi_t(\boldsymbol{\xi})
=\mathcal T_tv(\boldsymbol{\xi})-C_v\mathcal T_t\psi_t(\boldsymbol{\xi})
=\lambda_0(t)\int_K g_t(\boldsymbol{\xi}-\boldsymbol{\eta})
[v(\boldsymbol{\eta})-C_v\psi_t(\boldsymbol{\eta})]\,
m(\boldsymbol{\eta})d\boldsymbol{\eta}=0.
\]
That is,
\(
v=C_v\psi_t\, \text{a.e. on } K
\,\text{and}\,
\ker(I-\mathcal T_t)=\operatorname{span}\{\psi_t\}.
\)
By~\eqref{eq:short-range-annihilation}, the restriction $\left.\mathcal T_t\right|_{\ker\ell_t}$ maps
$\ker\ell_t$ into itself. Since $\mathcal T_t$ is compact on $C(K)$
and $\ker\ell_t$ is a closed subspace of $C(K)$, this restriction is
also compact on $\ker\ell_t$.
Since $\ker(I-\mathcal T_t)=\operatorname{span}\{\psi_t\}$
and $\ell_t(\psi_t)=J_t>0$,
\[
\ker\!\left(\left.(I-\mathcal T_t)\right|_{\ker\ell_t}\right)
=\operatorname{span}\{\psi_t\}\cap\ker\ell_t=\{0\}.
\]
Thus $I-\left.\mathcal T_t\right|_{\ker\ell_t}=\left.(I-\mathcal T_t)\right|_{\ker\ell_t}$ is injective.
By the Fredholm alternative~\cite[Theorem~3.4]{Kress1999},
$\left.(I-\mathcal T_t)\right|_{\ker\ell_t}=I-\left.\mathcal T_t\right|_{\ker\ell_t}$ is surjective and has a
bounded inverse.

By~\eqref{eq:short-range-annihilation}, $I-\mathcal T_t$ maps $C(K)$ into
$\ker\ell_t$. Since $\left.(I-\mathcal T_t)\right|_{\ker\ell_t}$ is
invertible on $\ker\ell_t$,
\[
C(K)=\ker\ell_t\oplus\ker(I-\mathcal T_t)
    =\ker\ell_t\oplus\operatorname{span}\{\psi_t\}.
\]
Equation~\eqref{eq:short-range-resonant-projection} gives
\[
(I-\mathcal T_t+\Pi_t)\psi_t=\psi_t,\qquad
(I-\mathcal T_t+\Pi_t)w=(I-\mathcal T_t)w
\quad (w\in\ker\ell_t).
\]
Therefore $I-\mathcal T_t+\Pi_t$ has a bounded inverse on $C(K)$.

By homogeneity, $g_t(\boldsymbol r)=|\boldsymbol r|^{-1}
g_t(\boldsymbol r/|\boldsymbol r|)$ for $\boldsymbol r\ne0$.
Since \(\gamma^{-2}|\boldsymbol{r}|^2+(\boldsymbol{\beta}\cdot\boldsymbol{r})^2\) on the unit sphere is bounded below by \(\min_{[0,T]}(1-|\boldsymbol\beta|^2)>0\), and \(\boldsymbol\beta, \lambda_0\) are \(C^3\), the angular factor \(t\mapsto\lambda_0(t)g_t|_{\mathbb S^2}\) is \(C^3\) in \(C(\mathbb S^2)\).
Since the map sending \(a\in C(\mathbb S^2)\) to the integral operator with kernel
\[
\frac{a((\boldsymbol\xi-\boldsymbol\eta)/
	|\boldsymbol\xi-\boldsymbol\eta|)}
{|\boldsymbol\xi-\boldsymbol\eta|}\,m(\boldsymbol\eta)
\]
is linear and bounded in operator norm by
\(\sup_{\boldsymbol\xi\in K}\int_K
m(\boldsymbol\eta)/|\boldsymbol\xi-\boldsymbol\eta|\,d\boldsymbol\eta
<\infty\), it follows that \(t\mapsto\mathcal T_t\) is \(C^3\) in operator norm.

The linear map sending $\psi_t$ to $\ell_t$ is bounded, since
\[
\|\ell_s-\ell_t\|_{C(K)^*}
\le \|m\|_{L^1(K)}\|\psi_s-\psi_t\|_\infty.
\]
Hence $\psi_t\in C^3([0,T];C(K))$ gives $C^3$ dependence of $\ell_t$
in the dual norm. Since $J_t=\ell_t(\psi_t)>0$,
~\eqref{eq:short-range-resonant-projection} gives
$C^3$ dependence of $\Pi_t$ in operator norm.

Inversion is analytic on the group of invertible elements of a unital
Banach algebra~\cite[Chapter~III, \S\,3.9]{BourbakiLie1989}.
Applying inversion in $\mathcal L(C(K),C(K))$ to the $C^3$ family
$I-\mathcal T_t+\Pi_t$ gives $\mathcal R\in C^3([0,T];\mathcal L(C(K),C(K)))$.
\end{proof}

\begin{theorem}\label{thm:short-range-point-limit}
With $\lambda_\varepsilon(t)=\nu_t^{-1}-\varepsilon\alpha(t)/J_t$,
the solutions of~\eqref{eq:short-range-model} converge in
$\mathcal D'(\mathbb R^3\times(0,T))$ to
$u=u^{\mathrm{in}}+E*(q(t)\delta_{\boldsymbol{y}(t)})$.
The scattering strength $q$ satisfies
\begin{equation}\label{eq:short-range-response}
\frac{\gamma^2}{4\pi c}q'
+\left(\alpha(t)+\frac{\gamma^4\boldsymbol{y}'\cdot\boldsymbol{y}''}{4\pi c^3}\right)q
=u^{\mathrm{in}}(\boldsymbol{y}(t),t),\qquad q(0)=0.
\end{equation}
\end{theorem}

\begin{proof}
We first consider~\eqref{eq:short-range-inner} and determine the asymptotic
form of its solution \(P_\varepsilon\) as \(\varepsilon\to0\).
To this end, we begin with smooth \(\alpha\) and replace \(F_\varepsilon\)
by a smooth function \(F(t)\) independent of \(\varepsilon\) and \(\boldsymbol\xi\).
We then extend the argument to general continuous \(\alpha\) and to
\(F_\varepsilon\) generated by the point emitter.
Once the analysis of \(P_\varepsilon\) is complete, we deduce the convergence
of \(u_\varepsilon\).

\smallskip\noindent\textbf{Step 1. Smooth $\alpha$ and smooth scalar forcing.}

First take smooth $\alpha$ and a smooth scalar forcing $F(t)$ vanishing
near $t=0$ in place of $F_\varepsilon$ in~\eqref{eq:short-range-inner}.
We use Lemmas~\ref{lem:short-range-operator-expansion}
and~\ref{lem:short-range-inverse-estimate} to construct and control
an approximation to $P_\varepsilon$.

Since $\mathcal T_t\psi_t=\psi_t$, we have
$(I-\mathcal T_t)(q\psi_t)=0$ for any scalar function $q(t)$.
We seek an approximation
\[
P_\varepsilon^{\rm app}=q\psi_t+\varepsilon w_1\in C^2([-1,T];C(K)),
\]
with $q$ and the spatial profile $w_1(\cdot,t)\in C(K)$ to be determined.
For profiles whose zero extensions have the $C^2$ regularity required by
Lemma~\ref{lem:short-range-operator-expansion}, the residual is
\begin{equation}\label{eq:short-range-corrector-residual}
\begin{aligned}
(I-\mathcal T_\varepsilon)P_\varepsilon^{\rm app}-\varepsilon F
&=(I-\mathcal T_t)(q\psi_t)\\
&\quad+\varepsilon\bigl[(I-\mathcal T_t)w_1
          -\mathcal S_t(q\psi_t)-F\bigr]+O(\varepsilon^2)\\
&=\varepsilon\bigl[(I-\mathcal T_t)w_1
          -\mathcal S_t(q\psi_t)-F\bigr]+O(\varepsilon^2).
\end{aligned}
\end{equation}
To cancel the coefficient of $\varepsilon$, we require
\begin{equation}\label{eq:short-range-corrector}
(I-\mathcal T_t)w_1=F+\mathcal S_t(q\psi_t).
\end{equation}
By~\eqref{eq:short-range-annihilation}, the right-hand side of
\eqref{eq:short-range-corrector} must satisfy
\[
\ell_t\bigl(F+\mathcal S_t(q\psi_t)\bigr)=0.
\]
We compute the left-hand side of this
solvability condition to determine $q$.
For $\boldsymbol n$ and $-\boldsymbol n$, $D$ is unchanged and $b$
changes sign. Thus the even part of $H_t$ is $h=\gamma^2/(4\pi c)$.
For $G_t$, the term $\boldsymbol a\cdot\boldsymbol n$ is odd, and
\[
(D+b)^2(2D-b)+(D-b)^2(2D+b)=4D^3.
\]
It follows that the even part of $G_t$ is
$j=\gamma^4\boldsymbol{y}'\cdot\boldsymbol{y}''/(4\pi c^3)$.
Under $(\boldsymbol{\xi},\boldsymbol{\eta})\mapsto
(-\boldsymbol{\xi},-\boldsymbol{\eta})$, the domain $K^2$ is unchanged and
$\boldsymbol n\mapsto-\boldsymbol n$. Using the evenness of $m$ and
$\psi_t$, we obtain
\begin{equation}\label{eq:short-range-H-double-integral-identity}
\begin{aligned}
&\int_{K^2}
 m(\boldsymbol{\xi})\psi_t(\boldsymbol{\xi})
 m(\boldsymbol{\eta})\partial_t(\lambda_0q\psi_t)(\boldsymbol{\eta})
 H_t(\boldsymbol n)\,d\boldsymbol{\eta}\,d\boldsymbol{\xi}\\
&=\int_{K^2}
 m(\boldsymbol{\xi})\psi_t(\boldsymbol{\xi})
 m(\boldsymbol{\eta})\partial_t(\lambda_0q\psi_t)(\boldsymbol{\eta})
 H_t(-\boldsymbol n)\,d\boldsymbol{\eta}\,d\boldsymbol{\xi}\\
&=\int_{K^2}
 m(\boldsymbol{\xi})\psi_t(\boldsymbol{\xi})
 m(\boldsymbol{\eta})\partial_t(\lambda_0q\psi_t)(\boldsymbol{\eta})
 \frac{H_t(\boldsymbol n)+H_t(-\boldsymbol n)}{2}\,
 d\boldsymbol{\eta}\,d\boldsymbol{\xi}\\
&=h\left(\int_K m(\boldsymbol{\xi})\psi_t(\boldsymbol{\xi})\,d\boldsymbol{\xi}\right)
 \left(\int_K m(\boldsymbol{\eta})\partial_t(\lambda_0q\psi_t)(\boldsymbol{\eta})\,
 d\boldsymbol{\eta}\right),
\end{aligned}
\end{equation}
and
\begin{equation}\label{eq:short-range-G-double-integral-identity}
\begin{aligned}
&\int_{K^2}
 m(\boldsymbol{\xi})\psi_t(\boldsymbol{\xi})
 m(\boldsymbol{\eta})\psi_t(\boldsymbol{\eta})G_t(\boldsymbol n)\,
 d\boldsymbol{\eta}\,d\boldsymbol{\xi}\\
&=\int_{K^2}
 m(\boldsymbol{\xi})\psi_t(\boldsymbol{\xi})
 m(\boldsymbol{\eta})\psi_t(\boldsymbol{\eta})G_t(-\boldsymbol n)\,
 d\boldsymbol{\eta}\,d\boldsymbol{\xi}\\
&=\int_{K^2}
 m(\boldsymbol{\xi})\psi_t(\boldsymbol{\xi})
 m(\boldsymbol{\eta})\psi_t(\boldsymbol{\eta})
 \frac{G_t(\boldsymbol n)+G_t(-\boldsymbol n)}{2}\,
 d\boldsymbol{\eta}\,d\boldsymbol{\xi}\\
&=j\left(\int_K m(\boldsymbol{\xi})\psi_t(\boldsymbol{\xi})\,
 d\boldsymbol{\xi}\right)^2 .
\end{aligned}
\end{equation}
At fixed $t$, the term containing $g_t$ in
\eqref{eq:short-range-first-order-operator} with $w=q\psi_t$ is
\begin{align*}
&\int_K m(\boldsymbol{\eta})\lambda_1(t)
 g_t(\boldsymbol{\xi}-\boldsymbol{\eta})q(t)\psi_t(\boldsymbol{\eta})
 \,d\boldsymbol{\eta}\\
&\qquad=\lambda_1(t)q(t)\int_K
 g_t(\boldsymbol{\xi}-\boldsymbol{\eta})m(\boldsymbol{\eta})
 \psi_t(\boldsymbol{\eta})\,d\boldsymbol{\eta}
 =\lambda_1(t)q(t)(\widetilde{\mathcal T}_t\psi_t)(\boldsymbol{\xi}).
\end{align*}
Using~\eqref{eq:short-range-first-order-operator},
\eqref{eq:short-range-H-double-integral-identity}, and
\eqref{eq:short-range-G-double-integral-identity} in the definition of $\ell_t$ gives
\begin{align*}
\ell_t(F+\mathcal S_t(q\psi_t))
&=\int_K m(\boldsymbol{\xi})\psi_t(\boldsymbol{\xi})
  \bigl[F(t)+(\mathcal S_t(q\psi_t))(\boldsymbol{\xi},t)\bigr]\,
  d\boldsymbol{\xi}\\
&=F(t)\int_K m(\boldsymbol{\xi})\psi_t(\boldsymbol{\xi})\,d\boldsymbol{\xi}\\
&\quad+\lambda_1 q\int_K m(\boldsymbol{\xi})\psi_t(\boldsymbol{\xi})
  (\widetilde{\mathcal T}_t\psi_t)(\boldsymbol{\xi})\,d\boldsymbol{\xi}\\
&\quad-h\left(\int_K m(\boldsymbol{\xi})\psi_t(\boldsymbol{\xi})\,d\boldsymbol{\xi}\right)
  \left(\int_K m(\boldsymbol{\eta})
  \partial_t(\lambda_0q\psi_t)(\boldsymbol{\eta})\,d\boldsymbol{\eta}\right)\\
&\quad-j\lambda_0q
  \left(\int_K m(\boldsymbol{\xi})\psi_t(\boldsymbol{\xi})\,d\boldsymbol{\xi}\right)^2 .
\end{align*}
By
\eqref{eq:short-range-tuning},
we have $\int_Km\psi_t=\nu_t$ and
$\widetilde{\mathcal T}_t\psi_t=\nu_t\psi_t$. Hence
\begin{align*}
\int_K m(\boldsymbol{\xi})\psi_t(\boldsymbol{\xi})
  (\widetilde{\mathcal T}_t\psi_t)(\boldsymbol{\xi})\,d\boldsymbol{\xi}
&=\frac{1}{\lambda_0}\int_Km(\boldsymbol{\xi})\psi_t(\boldsymbol{\xi})^2
  \,d\boldsymbol{\xi}=\frac{J_t}{\lambda_0},\\
\int_K m(\boldsymbol{\eta})\partial_t(\lambda_0q\psi_t)(\boldsymbol{\eta})\,
  d\boldsymbol{\eta}
&=\frac{d}{dt}\left(\lambda_0 q\int_Km(\boldsymbol{\eta})
  \psi_t(\boldsymbol{\eta})\,d\boldsymbol{\eta}\right)
 =\frac{d}{dt}(\lambda_0q/\lambda_0)=q'.
\end{align*}
Substitution into the projection, with $\lambda_1J_t=-\alpha(t)$,
yields
\begin{align*}
\ell_t(F+\mathcal S_t(q\psi_t))
&=\frac{F}{\lambda_0}+\frac{\lambda_1qJ_t}{\lambda_0}
   -\frac{hq'}{\lambda_0}-j\lambda_0q\left(\frac{1}{\lambda_0}\right)^2\\
&=\frac{F}{\lambda_0}-\frac{\alpha(t)q}{\lambda_0}
   -\frac{hq'}{\lambda_0}-\frac{jq}{\lambda_0}\\
&=\frac{F-hq'-(\alpha(t)+j)q}{\lambda_0}.
\end{align*}
The compatibility equation is therefore~\eqref{eq:short-range-response}
with right-hand side $F$.

Choose $q$ to solve
\[
\frac{\gamma^2}{4\pi c}q'
+\left(\alpha(t)+\frac{\gamma^4\boldsymbol{y}'\cdot\boldsymbol{y}''}{4\pi c^3}\right)q
=F(t),\qquad q(0)=0.
\]
Then $\ell_t(F+\mathcal S_t(q\psi_t))=0$. By
Lemma~\ref{lem:short-range-frozen-solvability}, choose $w_1(\cdot,t)$
as the unique solution of~\eqref{eq:short-range-corrector} in $\ker\ell_t$,
given by
\begin{equation}\label{eq:short-range-corrector-representation}
w_1(\cdot,t)=(I-\mathcal T_t+\Pi_t)^{-1}
\bigl[F(t)+(\mathcal S_t(q\psi_t))(\cdot,t)\bigr].
\end{equation}

Next, we show that the regularity of \(w_1\) and \(P_\varepsilon^{\rm app}\) is sufficient for the expansion~\eqref{eq:short-range-corrector-residual}. Since \(\boldsymbol{y}\in C^4\) and \(\alpha\) is smooth,
\(\gamma^2/(4\pi c)\) is \(C^3\) and bounded away from zero, and
\(\alpha(t)+\gamma^4\boldsymbol{y}'\cdot\boldsymbol{y}''/(4\pi c^3)\)
is \(C^2\). Thus the ODE gives \(q\in C^3([0,T])\). By
Lemma~\ref{lem:short-range-principal-eigenpair},
\(t\mapsto\psi_t\) belongs to \(C^3([0,T];C(K))\), so
\(t\mapsto q(t)\psi_t\) also belongs to \(C^3([0,T];C(K))\).
From~\eqref{eq:short-range-first-order-operator}, one can show that
\(t\mapsto(\mathcal S_t(q\psi_t))(\cdot,t)\) belongs to
\(C^2([0,T];C(K))\). Since \(\mathcal R\) is \(C^3\) by
Lemma~\ref{lem:short-range-frozen-solvability},
\eqref{eq:short-range-corrector-representation} gives
\(w_1\in C^2([0,T];C(K))\).

Since \(F=0\) near \(t=0\), ODE uniqueness gives \(q=0\) there, and
\eqref{eq:short-range-first-order-operator} and
\eqref{eq:short-range-corrector-representation}
give \(w_1=0\) there as well. Extending \(q\psi_t\) and \(w_1\) by zero for
\(t<0\) therefore gives \(P_\varepsilon^{\rm app}\in C^2([-1,T];C(K))\).

Substituting~\eqref{eq:short-range-corrector} into
\eqref{eq:short-range-corrector-residual} gives
\[
\bigl\|(I-\mathcal T_\varepsilon)P_\varepsilon^{\rm app}
-\varepsilon F\bigr\|_\infty\le C\varepsilon^2.
\]
By~\eqref{eq:short-range-inner} with $F_\varepsilon$ replaced by $F$,
\[
(I-\mathcal T_\varepsilon)(P_\varepsilon-P_\varepsilon^{\rm app})
=\varepsilon F-(I-\mathcal T_\varepsilon)P_\varepsilon^{\rm app}.
\]
Applying~\eqref{eq:short-range-inverse-supremum-bound} gives
\[
\|P_\varepsilon-P_\varepsilon^{\rm app}\|_\infty
\le \frac{C}{\varepsilon}\|(I-\mathcal T_\varepsilon)P_\varepsilon^{\rm app}
-\varepsilon F\|_\infty
\le C\varepsilon .
\]
Finally, $P_\varepsilon^{\rm app}-q\psi_t=\varepsilon w_1$, so
\[
\|P_\varepsilon-q\psi_t\|_\infty
\le \|P_\varepsilon-P_\varepsilon^{\rm app}\|_\infty
    +\varepsilon\|w_1\|_\infty
\le C\varepsilon .
\]

\smallskip\noindent\textbf{Step 2. Continuous $\alpha$ and smooth scalar forcing.}

For continuous nonnegative $\alpha$, take smooth nonnegative
$\alpha_k$ with $\|\alpha-\alpha_k\|_\infty\to0$. Use a superscript
$[\alpha]$ or $[\alpha_k]$ to indicate the parameter in the operator
$\mathcal T_\varepsilon$, inner profile $P_\varepsilon$ and response $q$.
For a bounded function $w$ extended by zero to $t<0$,
\begin{align*}
&((\mathcal T_\varepsilon^{[\alpha]}-
 \mathcal T_\varepsilon^{[\alpha_k]})w)(\boldsymbol{\xi},t)\\
&\quad=-\int_{K\cap\{\theta\ge0\}}
 \frac{\varepsilon^2(\alpha(\theta)-\alpha_k(\theta))m(\boldsymbol{\eta})}
 {4\pi R_\varepsilon J_\varepsilon J_\theta}
 w(\boldsymbol{\eta},\theta)\,d\boldsymbol{\eta}.
\end{align*}
By~\eqref{eq:short-range-delay-bounds},
\eqref{eq:short-range-retarded-distance-jacobian}, and the subsonic speed bound,
$\varepsilon/(R_\varepsilon J_\varepsilon)
\le C/|\boldsymbol{\xi}-\boldsymbol{\eta}|$. Since $J_t$ is bounded
away from zero and
$\sup_{\boldsymbol{\xi}\in K}\int_K
m(\boldsymbol{\eta})/|\boldsymbol{\xi}-\boldsymbol{\eta}|\,d\boldsymbol{\eta}<\infty$,
taking the supremum gives
\[
\|\mathcal T_\varepsilon^{[\alpha]}-
\mathcal T_\varepsilon^{[\alpha_k]}\|_{\infty\to\infty}
\le C\varepsilon\|\alpha-\alpha_k\|_\infty .
\]
With the same scalar forcing $F$, the two inner equations are
\begin{equation}\label{eq:short-range-inner-continuous-alpha}
(I-\mathcal T_\varepsilon^{[\alpha]})P_\varepsilon^{[\alpha]}
=\varepsilon F,
\end{equation}
\begin{equation}\label{eq:short-range-inner-smooth-alpha}
(I-\mathcal T_\varepsilon^{[\alpha_k]})P_\varepsilon^{[\alpha_k]}
=\varepsilon F.
\end{equation}
Subtracting~\eqref{eq:short-range-inner-smooth-alpha} from
\eqref{eq:short-range-inner-continuous-alpha} yields
\[
P_\varepsilon^{[\alpha]}-P_\varepsilon^{[\alpha_k]}
=(I-\mathcal T_\varepsilon^{[\alpha]})^{-1}
 (\mathcal T_\varepsilon^{[\alpha]}-\mathcal T_\varepsilon^{[\alpha_k]})
 P_\varepsilon^{[\alpha_k]} .
\]
Since $\sup_k\|\alpha_k\|_\infty<\infty$,
Lemma~\ref{lem:short-range-inverse-estimate} gives
$\|P_\varepsilon^{[\alpha_k]}\|_\infty\le C\|F\|_\infty$ with $C$
independent of $k$ for all $k$ and all sufficiently small $\varepsilon$, with the admissible range of $\varepsilon$ independent of $k$.
Hence
\[
\|P_\varepsilon^{[\alpha]}-P_\varepsilon^{[\alpha_k]}\|_\infty
\le \frac{C}{\varepsilon}\,
\varepsilon\|\alpha-\alpha_k\|_\infty
     \|P_\varepsilon^{[\alpha_k]}\|_\infty
\le C\|\alpha-\alpha_k\|_\infty .
\]
For the response, the difference satisfies
\[
h(q^{[\alpha]}-q^{[\alpha_k]})'
+(\alpha+j(t))(q^{[\alpha]}-q^{[\alpha_k]})
=(\alpha_k-\alpha)q^{[\alpha_k]} .
\]
Since $(q^{[\alpha]}-q^{[\alpha_k]})(0)=0$, the integrating factor gives
\begin{equation}\label{eq:short-range-response-integrating-factor}
q^{[\alpha]}(t)-q^{[\alpha_k]}(t)=\int_0^t
\exp\!\left(-\int_s^t\frac{\alpha(r)+j(r)}{h(r)}\,dr\right)
\frac{(\alpha_k(s)-\alpha(s))q^{[\alpha_k]}(s)}{h(s)}\,ds .
\end{equation}
Since $h=\gamma^2/(4\pi c)$ is bounded away from zero and $\alpha_k$ is uniformly bounded,
the integrating factor and~\eqref{eq:short-range-response-integrating-factor} give
$\|q^{[\alpha_k]}\|_\infty\le C\|F\|_{L^1(0,T)}$ and
$\|q^{[\alpha]}-q^{[\alpha_k]}\|_\infty\le C\|\alpha-\alpha_k\|_\infty$.
Combining these estimates with the $O(\varepsilon)$ bound for
$\alpha_k$ gives
\[
\|P_\varepsilon^{[\alpha]}-q^{[\alpha]}\psi_t\|_\infty
\le C\|\alpha-\alpha_k\|_\infty+C_k\varepsilon .
\]
First letting $\varepsilon\downarrow0$ and then $k\to\infty$ gives
\begin{equation}\label{eq:short-range-smooth-drive-convergence}
\|P_\varepsilon^{[\alpha]}-q^{[\alpha]}\psi_t\|_\infty\longrightarrow0
\end{equation}
for smooth scalar forcing functions vanishing near $t=0$ and continuous
nonnegative $\alpha$.

\smallskip\noindent\textbf{Step 3. Point-emitter forcing $F_\varepsilon$.}

For the point emitter, set $F(t)=u^{\mathrm{in}}(\boldsymbol{y}(t),t)$.
Assumption~\ref{ass:trajectory-separation}
in Section~\ref{sec:problem-formulation} gives
\[
d_T:=\min_{s,t\in[0,T]}|\boldsymbol{y}(t)-\boldsymbol{z}(s)|>0,
\]
since the two trajectory images on $[0,T]$ are compact and disjoint.
Choose $\varepsilon_0>0$ with
$\varepsilon_0\sup_{\boldsymbol{\xi}\in K}|\boldsymbol{\xi}|\le d_T/2$.
For $0<\varepsilon\le\varepsilon_0$ and $(\boldsymbol{\xi},t)\in K\times[0,T]$
where the field has arrived, this gives
$R_{\boldsymbol z}(\boldsymbol{y}(t)+\varepsilon\boldsymbol{\xi},t)\ge d_T/2$.
The speed bound in Assumption~\ref{ass:subsonic-trajectories}
of Section~\ref{sec:problem-formulation} gives
$J_{\boldsymbol z}\ge 1-c_0/c$ at those points. Since $\varphi$ is
continuous on $[0,T]$ by Assumption~\ref{ass:emitter-strength}
of Section~\ref{sec:problem-formulation}, and the field
vanishes before arrival,~\eqref{eq:incident-field} yields
\[
\sup_{0<\varepsilon\le\varepsilon_0}
\sup_{(\boldsymbol{\xi},t)\in K\times[0,T]}
|F_\varepsilon(\boldsymbol{\xi},t)|
\le \frac{\|\varphi\|_{L^\infty(0,T)}}{2\pi d_T(1-c_0/c)}.
\]
At the same time $t$, the arguments of $H$ in~\eqref{eq:incident-field}
at $\boldsymbol{y}(t)+\varepsilon\boldsymbol{\xi}$ and $\boldsymbol{y}(t)$
satisfy
\begin{equation}\label{eq:short-range-arrival-argument-difference}
\left|
 \left(t-\frac{|\boldsymbol{y}(t)+\varepsilon\boldsymbol{\xi}
                  -\boldsymbol{z}(0)|}{c}\right)
 -\left(t-\frac{|\boldsymbol{y}(t)-\boldsymbol{z}(0)|}{c}\right)
\right|\le \frac{\varepsilon|\boldsymbol{\xi}|}{c}.
\end{equation}
By Assumption~\ref{ass:subsonic-trajectories}
in Section~\ref{sec:problem-formulation},
\[
\frac{d}{dt}\left(t-\frac{|\boldsymbol{y}(t)-\boldsymbol{z}(0)|}{c}\right)
=1-\frac{\boldsymbol{y}'(t)\cdot(\boldsymbol{y}(t)-\boldsymbol{z}(0))}
{c|\boldsymbol{y}(t)-\boldsymbol{z}(0)|}
\ge 1-c_0/c.
\]
Integrating this derivative bound between $T_0$ and $t$ gives
\[
\left|t-\frac{|\boldsymbol{y}(t)-\boldsymbol{z}(0)|}{c}\right|
\ge(1-c_0/c)|t-T_0|.
\]
By~\eqref{eq:short-range-arrival-argument-difference}, the two arguments
of $H$ have the same sign whenever
$|t-T_0|>\varepsilon\sup_{\boldsymbol{\xi}\in K}|\boldsymbol{\xi}|/(c-c_0)$.
Suppose that both arguments of $H$ in
\eqref{eq:short-range-arrival-argument-difference} are nonnegative.
Write $\tau=\tau_1(\boldsymbol{y}(t),t)$ and
$\tau_\varepsilon=\tau_1(\boldsymbol{y}(t)+\varepsilon\boldsymbol{\xi},t)$.
The retarded-time equations in Section~\ref{sec:incident-field} give
\begin{align*}
c(t-\tau)&=|\boldsymbol{y}(t)-\boldsymbol{z}(\tau)|,\\
c(t-\tau_\varepsilon)&=
|\boldsymbol{y}(t)+\varepsilon\boldsymbol{\xi}
-\boldsymbol{z}(\tau_\varepsilon)|.
\end{align*}
By Assumption~\ref{ass:subsonic-trajectories} in
Section~\ref{sec:problem-formulation}, for $s\in[0,t]$,
\[
\frac{d}{ds}\bigl(c(t-s)-|\boldsymbol{y}(t)-\boldsymbol{z}(s)|\bigr)
=-c+\frac{\boldsymbol{z}'(s)\cdot(\boldsymbol{y}(t)-\boldsymbol{z}(s))}
{|\boldsymbol{y}(t)-\boldsymbol{z}(s)|}
\le -(c-c_0).
\]
Applying the mean-value theorem between $\tau$ and $\tau_\varepsilon$
and using the two retarded-time equations gives
\begin{equation}\label{eq:short-range-incident-retarded-time-difference}
\begin{aligned}
(c-c_0)|\tau_\varepsilon-\tau|
&\le \bigl|c(t-\tau_\varepsilon)
       -|\boldsymbol{y}(t)-\boldsymbol{z}(\tau_\varepsilon)|\bigr|\\
&=\bigl||\boldsymbol{y}(t)+\varepsilon\boldsymbol{\xi}
        -\boldsymbol{z}(\tau_\varepsilon)|
        -|\boldsymbol{y}(t)-\boldsymbol{z}(\tau_\varepsilon)|\bigr|\\
&\le \varepsilon|\boldsymbol{\xi}|.
\end{aligned}
\end{equation}
The emission strength and emitter velocity are uniformly continuous,
and the distance and Jacobian are bounded below by positive constants.
It follows from~\eqref{eq:short-range-incident-retarded-time-difference}
and~\eqref{eq:incident-field} that for every
fixed $\rho>0$ with $\{t\in[0,T]:|t-T_0|\ge\rho\}\ne\varnothing$,
\[
\sup_{\substack{\boldsymbol{\xi}\in K,\ t\in[0,T]\\ |t-T_0|\ge\rho}}
\bigl|F_\varepsilon(\boldsymbol{\xi},t)
-F(t)\bigr|
\longrightarrow 0
\qquad\text{as }\varepsilon\to0.
\]
The uniform bounds on $F$ and $F_\varepsilon$ control the $L^1$ gap
between constant upper and lower bounds on an arbitrarily short
neighborhood of the wavefront. Away from that neighborhood, $F$ is
continuous and $F_\varepsilon\to F$ uniformly in $\boldsymbol{\xi}\in K$,
so smooth upper and lower bounds can be chosen with an arbitrarily
small gap there. Both fields vanish on a common interval near $t=0$
for sufficiently small $\varepsilon$ by
\eqref{eq:short-range-arrival-argument-difference}.
Thus for every $\eta>0$ there exist $\varepsilon_\eta>0$ and smooth
functions $F_\eta^\pm$ vanishing near $t=0$ such that, for every
$0<\varepsilon\le\varepsilon_\eta$ and
$(\boldsymbol{\xi},t)\in K\times[0,T]$,
\[
F_\eta^-(t)\le F_\varepsilon(\boldsymbol{\xi},t)\le F_\eta^+(t),
\qquad F_\eta^-(t)\le F(t)\le F_\eta^+(t),
\]
and
\[
\|F_\eta^+-F_\eta^-\|_{L^1(0,T)}<\eta.
\]
For fixed $\eta$, let $P_\varepsilon^\pm$ and $q^\pm$ denote the
solutions with forcing $F_\eta^\pm$. By positivity of the inverse in
Lemma~\ref{lem:short-range-inverse-estimate},
$P_\varepsilon^-\le P_\varepsilon\le P_\varepsilon^+$ for
$0<\varepsilon\le\varepsilon_\eta$.
Applying the integrating factor to
$h(q^\pm-q)'+(\alpha+j)(q^\pm-q)=F_\eta^\pm-F$ with
$(q^\pm-q)(0)=0$ gives
$\|q^\pm-q\|_\infty\le C\|F_\eta^\pm-F\|_{L^1(0,T)}\le C\eta$.
Consequently,
\[
\|P_\varepsilon-q\psi_t\|_\infty
\le \max_{\pm}\|P_\varepsilon^\pm-q^\pm\psi_t\|_\infty
+\|\psi_t\|_\infty\max_{\pm}\|q^\pm-q\|_\infty.
\]
By~\eqref{eq:short-range-smooth-drive-convergence},
$\limsup_{\varepsilon\downarrow0}
\|P_\varepsilon-q\psi_t\|_\infty\le C\eta$.
Letting $\eta\downarrow0$ gives uniform convergence for the point emitter.

\smallskip\noindent\textbf{Step 4. Convergence of $u_\varepsilon$.}

Finally, define the potential contribution in~\eqref{eq:short-range-model} by
\[
V_\varepsilon(\boldsymbol{x},t)
=\frac{\lambda_\varepsilon(t)}{\varepsilon^2}
m\!\left(\frac{\boldsymbol{x}-\boldsymbol{y}(t)}{\varepsilon}\right)
u_\varepsilon(\boldsymbol{x},t),\qquad 0\le t\le T,
\]
and extend it by zero outside $[0,T]$. On $(0,T)$,
$u_\varepsilon-u^{\mathrm{in}}=E*V_\varepsilon$.
Changing variables $\boldsymbol{x}=\boldsymbol{y}(t)+
\varepsilon\boldsymbol{\eta}$ gives, for
$\Phi\in C_c^\infty(\mathbb R^3\times\mathbb R)$,
\begin{align*}
\langle V_\varepsilon,\Phi\rangle
={}&\int_0^T\lambda_\varepsilon(t)\int_K m(\boldsymbol{\eta})
 P_\varepsilon(\boldsymbol{\eta},t)
 \Phi(\boldsymbol{y}(t)+\varepsilon\boldsymbol{\eta},t)
 \,d\boldsymbol{\eta}\,dt\\
\longrightarrow{}&
\int_0^T \lambda_0(t)q(t)
 \left(\int_K m(\boldsymbol{\eta})\psi_t(\boldsymbol{\eta})
 \,d\boldsymbol{\eta}\right)\Phi(\boldsymbol{y}(t),t)\,dt\\
={}&\int_0^Tq(t)\Phi(\boldsymbol{y}(t),t)\,dt .
\end{align*}
The limit follows from $P_\varepsilon\to q\psi_t$ and
$\lambda_\varepsilon\to\lambda_0$ uniformly and from uniform continuity
of $\Phi$. The last
equality uses the normalization in~\eqref{eq:short-range-tuning}.
For $0<\varepsilon\le1$, the supports of $V_\varepsilon$ and its
limit are contained in the compact set
$\{(\boldsymbol{y}(t)+\boldsymbol{\eta},t):
\boldsymbol{\eta}\in K,\ 0\le t\le T\}$.
By continuity of convolution with a fixed distribution on sources
supported in a common compact set~\cite[Theorem~5.3.2(iv)]{FriedlanderJoshi1998},
$E*V_\varepsilon$ converges in distributions to the convolution of
$E$ with the limiting source \(q(t)\delta_{\boldsymbol{y}(t)} \). This proves the asserted convergence of $u_\varepsilon$.

\end{proof}

\section{Inversion algorithm}\label{sec:inversion-algorithm}
In this section we turn to the inverse problem of recovering the trajectory of a moving point scatterer.   
Define the functions 
\begin{equation}\label{eq:arrival-time-functions}
 g_{\boldsymbol{z}}(t,s):= t - s - \lvert \boldsymbol{y}(t) - \boldsymbol{z}(s) \rvert /c, 
\quad
 g_{\boldsymbol{x}}(t,s) := t - s - \lvert \boldsymbol{x}(t) - \boldsymbol{y}(s) \rvert /c,
 \end{equation}
  on \([0,\infty) \times [0,\infty)\). Here, \( \boldsymbol{y} \) denotes the trajectory of the point-like scatterer, \( \boldsymbol{z} \) denotes the trajectory of the point emitter, and \( \boldsymbol{x} \) denotes the trajectory of an observation point.

\begin{theorem} \label{thm:arrival-time-maps}
\( g_{\boldsymbol{z}}(t,s) \) and \( g_{\boldsymbol{x}}(t,s) \) satisfy the following properties:
\begin{enumerate}
    \item They are continuous, strictly increasing in \( t \), and strictly decreasing in \( s \).  
    \item There exist \( T_0>0 \) and \( T_{1,\boldsymbol{x}} >0\) such that \( g_{\boldsymbol{z}}(T_0, 0) = 0 \) and \( g_{\boldsymbol{x}}(T_{1,\boldsymbol{x}}, 0) = 0 \).  
    \item There exist unique strictly increasing, continuously differentiable functions \( h_{\boldsymbol{z}} \) and \( h_{\boldsymbol{x}} \) defined on \([0, \infty)\) and \([0, \infty)\), respectively, with \( h_{\boldsymbol{z}}(0) = T_0 \) and \( h_{\boldsymbol{x}}(0) = T_{1,\boldsymbol{x}} \), such that \( g_{\boldsymbol{z}}(h_{\boldsymbol{z}}(s), s) = 0 \) and \( g_{\boldsymbol{x}}(h_{\boldsymbol{x}}(s), s) = 0 \) hold for all \( s \geq 0 \).
\end{enumerate}
\end{theorem}

\begin{proof}
By~\eqref{eq:arrival-time-functions} and the continuity of the trajectories, \(g_{\boldsymbol z}\) and \(g_{\boldsymbol x}\) are continuous. Moreover, since
\begin{align*}
\frac{\partial g_{\boldsymbol{z}}}{\partial t}(t,s) = 1-\frac{\boldsymbol{y}'(t) \cdot (\boldsymbol{y}(t) - \boldsymbol{z}(s))}{c|\boldsymbol{y}(t) - \boldsymbol{z}(s)|}\geq 1-c_0/c >0, \\  
\frac{\partial g_{\boldsymbol{z}}}{\partial s}(t,s) =\frac{\boldsymbol{z}'(s) \cdot (\boldsymbol{y}(t) - \boldsymbol{z}(s))}{c|\boldsymbol{y}(t) - \boldsymbol{z}(s)|}-1\leq c_0/c-1 <0,
\end{align*}

\( g_{\boldsymbol{z}} \) is strictly increasing in \( t \) and strictly decreasing in \( s \). Moreover, \( g_{\boldsymbol{z}}(0,0) < 0 \) and \(\lim_{t \to \infty} g_{\boldsymbol{z}}(t,0) = \infty\). Hence, there exists a unique \( T_0 > 0 \) such that \( g_{\boldsymbol{z}}(T_0,0) = 0 \). By the implicit function theorem, there exists a unique continuously differentiable function \( h_{\boldsymbol{z}}: [0, \infty) \to [T_0, \infty) \) such that \( h_{\boldsymbol{z}}(0) = T_0 \) and \( g_{\boldsymbol{z}}(h_{\boldsymbol{z}}(s), s) = 0 \), with  

\[
h'_{\boldsymbol{z}}(s) = \left(-\frac{\partial g_{\boldsymbol{z}}}{\partial s} \middle/ \frac{\partial g_{\boldsymbol{z}}}{\partial t}\right)\Bigg |_{(t,s)=(h_{\boldsymbol{z}}(s),s)} > 0, \quad \forall s \geq 0.
\]
Therefore, \( h_{\boldsymbol{z}} \) is a strictly increasing function.

The properties of \( g_{\boldsymbol{x}} \) can be verified in a similar manner.
\end{proof}

The following corollary is immediate from Theorem~\ref{thm:arrival-time-maps}.
\begin{corollary}
	For each \(t\geq 0\), \(h_{\boldsymbol{x}}(t)\) is the unique solution \(h\geq 0\) of the equation
	\begin{equation}\label{eq:receiver-arrival-time}
		h-t-\frac{|\boldsymbol{x}(h)-\boldsymbol{y}(t)|}{c}=0.
	\end{equation}
\end{corollary}

Since \( h_{\boldsymbol{z}} \) and \( h_{\boldsymbol{x}} \) are strictly increasing and continuously differentiable, each possesses a well-defined inverse function: \( h_{\boldsymbol{z}}^{-1} : [T_0, \infty) \to [0, \infty) \) and \( h_{\boldsymbol{x}}^{-1} : [T_{1,\boldsymbol{x}}, \infty) \to [0, \infty) \), where \( T_0>0 \) and \( T_{1,\boldsymbol{x}} >0\) are specified in the second assertion of Theorem~\ref{thm:arrival-time-maps}.
Let \( T_{2,\boldsymbol{x}} := h_{\boldsymbol{x}}(T_0) > h_{\boldsymbol{x}}(0) = T_{1,\boldsymbol{x}}\).

\begin{theorem}\label{thm:measured-field}
For each \(t\geq 0\), the measured scattered field satisfies
\begin{equation}\label{eq:measured-field}
	U_{\boldsymbol{x}}(h_{\boldsymbol{x}}(t))
	=\frac{q(t)}{4\pi|\boldsymbol{x}(h_{\boldsymbol{x}}(t))-\boldsymbol{y}(t)|
		\left(1-\dfrac{\boldsymbol{y}'(t)\cdot(\boldsymbol{x}(h_{\boldsymbol{x}}(t))-\boldsymbol{y}(t))}{c|\boldsymbol{x}(h_{\boldsymbol{x}}(t))-\boldsymbol{y}(t)|}\right)}.
\end{equation}
Consequently, the following statements hold:
\begin{enumerate}
\item 
\(
T_{2,\boldsymbol{x}}=\inf\{t>0:U_{\boldsymbol{x}}(t)\neq 0\}.
\)
\item \(U_{\boldsymbol{x}}(t)>0\) for \(t>T_{2,\boldsymbol{x}}\), and \(q(t)>0\) for \(t>T_0\).
\item The right-hand derivative of \(U_{\boldsymbol{x}}\) at \(T_{2,\boldsymbol{x}}\) is given by
\begin{equation}\label{eq:wavefront-derivative}
U_{\boldsymbol{x}}'(T_{2,\boldsymbol{x}}+)
=\frac{q'(T_0+)}{4\pi|\boldsymbol{x}(T_{2,\boldsymbol{x}})-\boldsymbol{y}(T_0)|}
\frac{1-\dfrac{\boldsymbol{x}'(T_{2,\boldsymbol{x}})\cdot(\boldsymbol{x}(T_{2,\boldsymbol{x}})-\boldsymbol{y}(T_0))}{c|\boldsymbol{x}(T_{2,\boldsymbol{x}})-\boldsymbol{y}(T_0)|}}
{\left(1-\dfrac{\boldsymbol{y}'(T_0)\cdot(\boldsymbol{x}(T_{2,\boldsymbol{x}})-\boldsymbol{y}(T_0))}{c|\boldsymbol{x}(T_{2,\boldsymbol{x}})-\boldsymbol{y}(T_0)|}\right)^2},
\end{equation}
where \(q'(T_0+)>0\) denotes the right-hand derivative of \(q\) at \(T_0\).
\end{enumerate}
\end{theorem}

\begin{proof}
At the arrival time \(h_{\boldsymbol{x}}(t)\), the retarded time in~\eqref{eq:scattered-field} equals \(t\), since
\[
h_{\boldsymbol{x}}(t)-t
-\frac{|\boldsymbol{x}(h_{\boldsymbol{x}}(t))-\boldsymbol{y}(t)|}{c}=0.
\]
Substitution into~\eqref{eq:scattered-field} gives~\eqref{eq:measured-field}. At \(t=0\), the same identity follows from \(q(0)=0\).

By~\eqref{eq:incident-field} and Theorem~\ref{thm:arrival-time-maps}, 
the incident field at the scatterer \( u^{\mathrm{in}}(\boldsymbol y(t),t) \)  vanishes for \(t<T_0\) and is positive for \(t>T_0\). The integral representation~\eqref{eq:scattering-strength} implies
\[
q(t)=0\quad\text{for }0\leq t\leq T_0,
\qquad q(t)>0\quad\text{for }t>T_0.
\]
The subsonic condition ensures that 
\[
1-\dfrac{\boldsymbol{y}'(T_0)\cdot(\boldsymbol{x}(T_{2,\boldsymbol{x}})-\boldsymbol{y}(T_0))}{c|\boldsymbol{x}(T_{2,\boldsymbol{x}})-\boldsymbol{y}(T_0)|}>0.
\]
Since \(h_{\boldsymbol{x}}\) is strictly increasing and \(h_{\boldsymbol{x}}(T_0)=T_{2,\boldsymbol{x}}\), \(U_{\boldsymbol{x}}(t)\) vanishes for \(0\leq t\leq T_{2,\boldsymbol{x}}\) and is positive for \(t>T_{2,\boldsymbol{x}}\). This proves assertions~(1) and~(2).

To prove assertion~(3), set
\[
d(t)=|\boldsymbol{x}(h_{\boldsymbol{x}}(t))-\boldsymbol{y}(t)|,
\qquad
\boldsymbol{n}(t)=\frac{\boldsymbol{x}(h_{\boldsymbol{x}}(t))-\boldsymbol{y}(t)}{d(t)}.
\]
Differentiating~\eqref{eq:receiver-arrival-time} gives
\[
h_{\boldsymbol{x}}'(t)
=\frac{1-\boldsymbol{y}'(t)\cdot\boldsymbol{n}(t)/c}
{1-\boldsymbol{x}'(h_{\boldsymbol{x}}(t))\cdot\boldsymbol{n}(t)/c}>0.
\]
From~\eqref{eq:response-law} and \(q(T_0)=0\), the right-hand derivative exists and satisfies
\[
q'(T_0+)=\frac{4\pi c}{\gamma(T_0)^2}
u^{\mathrm{in}}(\boldsymbol{y}(T_0),T_0+)>0.
\]
Using \(U_{\boldsymbol{x}}(h_{\boldsymbol{x}}(T_0))=U_{\boldsymbol{x}}(T_{2,\boldsymbol{x}})=0\) and the monotonicity of \(h_{\boldsymbol{x}}\), we obtain
\begin{align*}
U_{\boldsymbol{x}}'(T_{2,\boldsymbol{x}}+)
&=\lim_{t\downarrow T_0}
\frac{U_{\boldsymbol{x}}(h_{\boldsymbol{x}}(t))}
{h_{\boldsymbol{x}}(t)-h_{\boldsymbol{x}}(T_0)}\\
&= \lim_{t\downarrow T_0}
\frac{q(t)}
{4\pi d(t) \bigl(1-\boldsymbol{y}'(t)\cdot\boldsymbol{n}(t)/c\bigr)\bigl(h_{\boldsymbol{x}}(t)-h_{\boldsymbol{x}}(T_0) \bigr)} \\
&=
\lim_{t\downarrow T_0}
\frac{\dfrac{q(t)-q(T_0)}{t-T_0}}
{4\pi d(t)\bigl(1-\boldsymbol{y}'(t)\cdot\boldsymbol{n}(t)/c\bigr)
	\dfrac{h_{\boldsymbol{x}}(t)-h_{\boldsymbol{x}}(T_0)}{t-T_0}}
\\
&=\frac{q'(T_0+)}
{4\pi d(T_0)\bigl(1-\boldsymbol{y}'(T_0)\cdot\boldsymbol{n}(T_0)/c\bigr)h_{\boldsymbol{x}}'(T_0)}\\
&=\frac{q'(T_0+)}{4\pi d(T_0)}
\frac{1-\boldsymbol{x}'(T_{2,\boldsymbol{x}})\cdot\boldsymbol{n}(T_0)/c}
{\bigl(1-\boldsymbol{y}'(T_0)\cdot\boldsymbol{n}(T_0)/c\bigr)^2}.
\end{align*}
Substituting the definitions of \(d(T_0)\) and \(\boldsymbol{n}(T_0)\) yields the claimed formula.
\end{proof}

Define the row vector \(\boldsymbol{a}_{\boldsymbol{x}}\in\mathbb{R}^{1\times 4}\) by
\[
\boldsymbol{a}_{\boldsymbol{x}}=
\left(
\frac{(\boldsymbol{x}(T_{2,\boldsymbol{x}})-\boldsymbol{y}(T_0))^{\mathsf T}}{|\boldsymbol{x}(T_{2,\boldsymbol{x}})-\boldsymbol{y}(T_0)|}\;,\qquad
\sqrt{
\frac{1-\dfrac{\boldsymbol{x}'(T_{2,\boldsymbol{x}})\cdot(\boldsymbol{x}(T_{2,\boldsymbol{x}})-\boldsymbol{y}(T_0))}{c|\boldsymbol{x}(T_{2,\boldsymbol{x}})-\boldsymbol{y}(T_0)|}}
{4\pi|\boldsymbol{x}(T_{2,\boldsymbol{x}})-\boldsymbol{y}(T_0)|\,U'_{\boldsymbol{x}}(T_{2,\boldsymbol{x}}+)}
}
\right).
\]
For \(t>T_0\), define the row vector \(\boldsymbol{b}_{\boldsymbol{x}}(t)\in\mathbb{R}^{1\times 4}\) by
\[
\boldsymbol{b}_{\boldsymbol{x}}(t)=
\left(
\frac{(\boldsymbol{x}(h_{\boldsymbol{x}}(t))-\boldsymbol{y}(t))^{\mathsf T}}{|\boldsymbol{x}(h_{\boldsymbol{x}}(t))-\boldsymbol{y}(t)|}\;,\qquad
	\frac{1}
	{4\pi|\boldsymbol{x}(h_{\boldsymbol{x}}(t))-\boldsymbol{y}(t)|\,U_{\boldsymbol{x}}(h_{\boldsymbol{x}}(t))}
\right).
\]

By~\eqref{eq:wavefront-derivative}, at \(t=T_0\), \(\boldsymbol{y}'(T_0)\) and \(q'(T_0+)\) satisfy
\begin{equation}\label{eq:initial-velocity-row}
	\boldsymbol{a}_{\boldsymbol{x}}
	\begin{pmatrix}
		\boldsymbol{y}'(T_0)/c\\[10pt]
		\sqrt{q'(T_0+)}
	\end{pmatrix}
	=1.
\end{equation}
By~\eqref{eq:measured-field}, for \(t>T_0\), \(\boldsymbol{y}'(t)\) and \(q(t)\) satisfy
\begin{equation}\label{eq:velocity-strength-row}
\boldsymbol{b}_{\boldsymbol{x}}(t)
\begin{pmatrix}
	\boldsymbol{y}'(t)/c\\[10pt]
	q(t)
\end{pmatrix}
=1.
\end{equation}

In what follows, let \(\boldsymbol{x}=\boldsymbol{x}_i\) for \(i=1,2,3,4\), and let \(\boldsymbol{1}\) denote the column vector of all ones.
Define the matrices \(A,B(t)\in\mathbb{R}^{4\times4}\), \(C\in\mathbb{R}^{5\times4}\), and the column vector \(\boldsymbol{c}\in\mathbb{R}^{5}\) by
\[
A=
\begin{pmatrix}
	\boldsymbol{a}_{\boldsymbol{x}_1}\\[10pt]
	\boldsymbol{a}_{\boldsymbol{x}_2}\\[10pt]
	\boldsymbol{a}_{\boldsymbol{x}_3}\\[10pt]
	\boldsymbol{a}_{\boldsymbol{x}_4}
\end{pmatrix},\qquad
B(t)=
\begin{pmatrix}
	\boldsymbol{b}_{\boldsymbol{x}_1}(t)\\[10pt]
	\boldsymbol{b}_{\boldsymbol{x}_2}(t)\\[10pt]
	\boldsymbol{b}_{\boldsymbol{x}_3}(t)\\[10pt]
	\boldsymbol{b}_{\boldsymbol{x}_4}(t)
\end{pmatrix},
\]
\[
C=\begin{pmatrix}
	2\boldsymbol{x}_1(T_{2,\boldsymbol{x}_1})^{\mathsf T} & -2cT_{2,\boldsymbol{x}_1}\\[6pt]
	2\boldsymbol{x}_2(T_{2,\boldsymbol{x}_2})^{\mathsf T} & -2cT_{2,\boldsymbol{x}_2}\\[6pt]
	2\boldsymbol{x}_3(T_{2,\boldsymbol{x}_3})^{\mathsf T} & -2cT_{2,\boldsymbol{x}_3}\\[6pt]
	2\boldsymbol{x}_4(T_{2,\boldsymbol{x}_4})^{\mathsf T} & -2cT_{2,\boldsymbol{x}_4}\\[6pt]
	2\boldsymbol{z}(0)^{\mathsf T} & 0
\end{pmatrix},\qquad
\boldsymbol{c}=\begin{pmatrix}
	|\boldsymbol{x}_1(T_{2,\boldsymbol{x}_1})|^2-c^2T_{2,\boldsymbol{x}_1}^2\\[6pt]
	|\boldsymbol{x}_2(T_{2,\boldsymbol{x}_2})|^2-c^2T_{2,\boldsymbol{x}_2}^2\\[6pt]
	|\boldsymbol{x}_3(T_{2,\boldsymbol{x}_3})|^2-c^2T_{2,\boldsymbol{x}_3}^2\\[6pt]
	|\boldsymbol{x}_4(T_{2,\boldsymbol{x}_4})|^2-c^2T_{2,\boldsymbol{x}_4}^2\\[6pt]
	|\boldsymbol{z}(0)|^2
\end{pmatrix}.
\]

\begin{corollary}
	\begin{enumerate}
		\item \(\boldsymbol{y}(T_0)\) and \(T_0\) satisfy
		\begin{equation}\label{eq:initial-localization}
			C
			\begin{pmatrix}
				\boldsymbol{y}(T_0)\\[10pt]
				cT_0
			\end{pmatrix}
			-
			\begin{pmatrix}
				\boldsymbol{y}(T_0)\\[10pt]
				-cT_0
			\end{pmatrix}^{\mathsf T}
			\begin{pmatrix}
				\boldsymbol{y}(T_0)\\[10pt]
				cT_0
			\end{pmatrix}
			\boldsymbol{1}
			=\boldsymbol{c},
		\end{equation}
		\item \(\boldsymbol{y}'(T_0)\) and \(q'(T_0+)\) satisfy
		\begin{equation}\label{eq:initial-velocity-system}
			A
			\begin{pmatrix}
				\boldsymbol{y}'(T_0)/c\\[10pt]
				\sqrt{q'(T_0+)}
			\end{pmatrix}
			=\boldsymbol{1}.
		\end{equation}
		\item For \(t>T_0\), \(\boldsymbol{y}'(t)\) and \(q(t)\) satisfy
		\begin{equation}\label{eq:velocity-strength-system}
		B(t)
		\begin{pmatrix}
			\boldsymbol{y}'(t)/c\\[10pt]
			q(t)
		\end{pmatrix}
		=\boldsymbol{1}.
	\end{equation}
	\end{enumerate}
	Moreover, the matrices \(A\) and \(B(t)\) are always invertible, and system~\eqref{eq:initial-localization} has at most two distinct solutions.
\end{corollary}

\begin{proof}
By Theorem~\ref{thm:arrival-time-maps}, we have
\begin{equation}\label{eq:initial-travel-distances}
	|\boldsymbol{y}(T_0)-\boldsymbol{z}(0)|=cT_0,
	\qquad |\boldsymbol{x}_i(T_{2,\boldsymbol{x}_i})-\boldsymbol{y}(T_0)|=cT_{2,\boldsymbol{x}_i}-cT_0, i=1,2,3,4.
\end{equation}
Squaring both sides yields~\eqref{eq:initial-localization}. Equations~\eqref{eq:initial-velocity-system} and~\eqref{eq:velocity-strength-system} follow by stacking the four row identities~\eqref{eq:initial-velocity-row} and~\eqref{eq:velocity-strength-row}.

To prove invertibility, let \(N_A\in\mathbb{R}^{4\times3}\) have rows
\((\boldsymbol{x}_i(T_{2,\boldsymbol{x}_i})-\boldsymbol{y}(T_0))^{\mathsf T}/|\boldsymbol{x}_i(T_{2,\boldsymbol{x}_i})-\boldsymbol{y}(T_0)|\), and set \(\boldsymbol{v}=\boldsymbol{y}'(T_0)/c\). The row identities~\eqref{eq:initial-velocity-row} give
\[
A=[N_A,\,(\boldsymbol{1}-N_A\boldsymbol{v})/\sqrt{q'(T_0+)}],
\qquad \det A=\frac{\det[N_A,\boldsymbol{1}]}{\sqrt{q'(T_0+)}}\neq0,
\]
where the last inequality follows from the noncoplanarity of the four unit directions. The same argument shows that \(B(t)\) is invertible.

Let
\[
N_C=\begin{pmatrix}
	2\left(
	\boldsymbol{x}_2(T_{2,\boldsymbol{x}_2})-\boldsymbol{x}_1(T_{2,\boldsymbol{x}_1})
	\right)^{\mathsf T}\\[6pt]
	2\left(
	\boldsymbol{x}_3(T_{2,\boldsymbol{x}_3})-\boldsymbol{x}_1(T_{2,\boldsymbol{x}_1})
	\right)^{\mathsf T}\\[6pt]
	2\left(
	\boldsymbol{x}_4(T_{2,\boldsymbol{x}_4})-\boldsymbol{x}_1(T_{2,\boldsymbol{x}_1})
	\right)^{\mathsf T}
\end{pmatrix}.
\]
By the assumptions, the matrix \(N_C\) is invertible.
Let
\[
\boldsymbol{a} = 
N_C^{-1}\begin{pmatrix}
	|\boldsymbol{x}_2(T_{2,\boldsymbol{x}_2})|^2-|\boldsymbol{x}_1(T_{2,\boldsymbol{x}_1})|^2-c^2T_{2,\boldsymbol{x}_2}^2+c^2T_{2,\boldsymbol{x}_1}^2\\[6pt]
	|\boldsymbol{x}_3(T_{2,\boldsymbol{x}_3})|^2-|\boldsymbol{x}_1(T_{2,\boldsymbol{x}_1})|^2-c^2T_{2,\boldsymbol{x}_3}^2+c^2T_{2,\boldsymbol{x}_1}^2\\[6pt]
	|\boldsymbol{x}_4(T_{2,\boldsymbol{x}_4})|^2-|\boldsymbol{x}_1(T_{2,\boldsymbol{x}_1})|^2-c^2T_{2,\boldsymbol{x}_4}^2+c^2T_{2,\boldsymbol{x}_1}^2
\end{pmatrix},
\]
\[
\boldsymbol{b} = 
N_C^{-1}\begin{pmatrix}
	2cT_{2,\boldsymbol{x}_2}-2cT_{2,\boldsymbol{x}_1}\\[6pt]
	2cT_{2,\boldsymbol{x}_3}-2cT_{2,\boldsymbol{x}_1}\\[6pt]
	2cT_{2,\boldsymbol{x}_4}-2cT_{2,\boldsymbol{x}_1}
\end{pmatrix},\quad
\sigma=cT_0.
\]
Subtracting the first receiver equation in~\eqref{eq:initial-localization} from the other three yields
\[
\boldsymbol{y}(T_0)=\boldsymbol{a}+\boldsymbol{b}\sigma.
\]
Substituting this into~\eqref{eq:initial-travel-distances} and squaring gives
\[
|\boldsymbol{a}-\boldsymbol{z}(0)+\boldsymbol{b}\sigma|^2-\sigma^2=0,
\]
\begin{equation}\label{eq:receiver-time-polynomial}
	|\boldsymbol{x}_i(T_{2,\boldsymbol{x}_i})-\boldsymbol{a}-\boldsymbol{b}\sigma|^2
	-(cT_{2,\boldsymbol{x}_i}-\sigma)^2=0,\qquad i=1,2,3,4.
\end{equation}
If \(|\boldsymbol{b}|\neq1\), each equation is quadratic in \(\sigma\), so they have at most two common solutions.
If \(|\boldsymbol{b}|=1\), all five equations have degree at most one. The coefficient of \(\sigma\) in~\eqref{eq:receiver-time-polynomial} is
\[
2\left[cT_{2,\boldsymbol{x}_i}
-\bigl(\boldsymbol{x}_i(T_{2,\boldsymbol{x}_i})-\boldsymbol{a}\bigr)\cdot\boldsymbol{b}\right].
\]
Note that
\begin{align*}
	cT_{2,\boldsymbol{x}_i}
	-\bigl(\boldsymbol{x}_i(T_{2,\boldsymbol{x}_i})-\boldsymbol{a}\bigr)\cdot\boldsymbol{b}
	&=cT_{2,\boldsymbol{x}_i}
	-\bigl(\boldsymbol{x}_i(T_{2,\boldsymbol{x}_i})-\boldsymbol{a}\bigr)\cdot\boldsymbol{b}
	-[\boldsymbol{a}+cT_0\boldsymbol{b}-\boldsymbol{y}(T_0)]\cdot\boldsymbol{b}\\
	&=cT_{2,\boldsymbol{x}_i}-cT_0
	-\bigl(\boldsymbol{x}_i(T_{2,\boldsymbol{x}_i})-\boldsymbol{y}(T_0)\bigr)\cdot\boldsymbol{b}\\
	&=|\boldsymbol{x}_i(T_{2,\boldsymbol{x}_i})-\boldsymbol{y}(T_0)|
	-\bigl(\boldsymbol{x}_i(T_{2,\boldsymbol{x}_i})-\boldsymbol{y}(T_0)\bigr)\cdot\boldsymbol{b}\geq0.
\end{align*}
If all four coefficients vanished, equality in the Cauchy--Schwarz inequality would place all four receiver positions on the line through \(\boldsymbol{y}(T_0)\) in direction \(\boldsymbol{b}\), contradicting noncoplanarity. Thus at least one equation has a nonzero linear coefficient, and the true solution is the unique common solution. In either case, there are at most two values of \(\sigma\), each determining a unique pair \((T_0,\boldsymbol{y}(T_0))\).
\end{proof}

\begin{theorem}
For \(t>T_0\), the scattering parameter is given by
\begin{equation}\label{eq:alpha-recovery}
\alpha(t)=\frac{\gamma(t)^2}{4\pi c\,p(t)}
\left(\frac{4\pi c}{\gamma(t)}u^{\mathrm{in}}(\boldsymbol y(t),t)-p'(t)\right),
\end{equation}
where \(p(t)=\gamma(t)q(t)\).
\end{theorem}

\begin{proof}
Since
\[
\gamma'(t)=\frac{\gamma(t)^3}{c^2}\boldsymbol{y}'(t)\cdot\boldsymbol{y}''(t),
\]
substituting \(p=\gamma q\) into~\eqref{eq:response-law} gives
\[
p'(t)+\frac{4\pi c\alpha(t)}{\gamma(t)^2}p(t)
=\frac{4\pi c}{\gamma(t)}u^{\mathrm{in}}(\boldsymbol{y}(t),t).
\]
For \(t>T_0\), \(p(t)=\gamma(t)q(t)>0\), so rearranging yields~\eqref{eq:alpha-recovery}.
\end{proof}

Using~\eqref{eq:initial-localization},~\eqref{eq:initial-velocity-system},
\eqref{eq:velocity-strength-system}, and~\eqref{eq:alpha-recovery}, we construct
an inversion algorithm for recovering the trajectory and scattering parameter of the point-like scatterer.

\begin{algorithm}[!ht]
\DontPrintSemicolon
\caption{Reconstruction of the scatterer trajectory and scattering parameter}\label{alg:trajectory-parameter-reconstruction}
\begin{minipage}{\dimexpr\linewidth-2\algomargin-6pt\relax}
\noindent\hspace*{\dimexpr-\labelsep-3.52pt\relax}\textbf{Input:} Measured signals \(U_{\boldsymbol{x}_i}\), receiver trajectories \(\boldsymbol{x}_i\), \(i=1,2,3,4\), emitter trajectory \(\boldsymbol{z}\) and strength \(\varphi\), wave speed \(c\), and final reconstruction time \(T\).\par
\noindent\hspace*{\dimexpr-\labelsep-3.52pt\relax}\textbf{Output:} Initial scattering time \(T_0\), trajectory \(\boldsymbol{y}(t)\), scattering strength \(q(t)\), and scattering parameter \(\alpha(t)\) for \(T_0\leq t\leq T\).\par
\begin{enumerate}
\item[\textbf{Step 1:}] Compute
\(
T_{2,\boldsymbol{x}_i}=\inf\{t>0:U_{\boldsymbol{x}_i}(t)\neq0\}, 
i=1,2,3,4.
\)
\item[\textbf{Step 2:}] Solve system~\eqref{eq:initial-localization} to obtain candidate pairs \((\boldsymbol{y}_0,T_0)\), and apply the subsequent steps to each pair.
\item[\textbf{Step 3:}] Solve the initial-value problem
\begin{equation}\label{eq:trajectory-ode}
	\begin{cases}
		\boldsymbol{y}'(t)=
		\begin{cases}
			cP A^{-1}\boldsymbol{1},&t=T_0,\\[4pt]
			cP B(t)^{-1}\boldsymbol{1},&T_0<t\leq T,
		\end{cases}\\[8pt]
		\boldsymbol{y}(T_0)=\boldsymbol{y}_0,
	\end{cases}
\end{equation}
where 
\[
P=\begin{pmatrix}
	1&0&0&0\\
	0&1&0&0\\
	0&0&1&0
\end{pmatrix}.
\]
Here \(B(t)\) is evaluated at the current position \(\boldsymbol{y}(t)\). The quantities \(h_{\boldsymbol{x}_1}\), \(h_{\boldsymbol{x}_2}\), \(h_{\boldsymbol{x}_3}\), \(h_{\boldsymbol{x}_4}\), required to form the rows of \(B\), are obtained by solving~\eqref{eq:receiver-arrival-time}.
\item[\textbf{Step 4:}] Set
\[
q(t)=
\begin{cases}
0,&t=T_0,\\[6pt]
\begin{pmatrix}0&0&0&1\end{pmatrix}
B(t)^{-1}\boldsymbol{1},&T_0<t\leq T.
\end{cases}
\]
\item[\textbf{Step 5:}] Recover \(\alpha(t)\) for \(T_0<t\leq T\) using~\eqref{eq:alpha-recovery}, and determine \(\alpha(T_0)\) by continuity.
\end{enumerate}
\end{minipage}
\end{algorithm}

\begin{remark}\label{rem:initial-localization-ambiguity}
	Step~2 of Algorithm~\ref{alg:trajectory-parameter-reconstruction} applies the subsequent steps to each candidate pair, since~\eqref{eq:initial-localization} may have two distinct solutions. The noncoplanarity conditions in Assumptions~\ref{ass:receiver-position-noncoplanarity} and~\ref{ass:receiver-direction-noncoplanarity} of Section~\ref{sec:problem-formulation} generally cannot exclude the second algebraic candidate. A sufficient and stronger geometric condition is the following: for any chosen observation times \(t_1,t_2,t_3,t_4>0\), the closure \(\overline{\Omega_1}\) lies strictly inside the nondegenerate tetrahedron with vertices \(\boldsymbol x_i(t_i)\). This condition implies the two noncoplanarity conditions and ensures that~\eqref{eq:initial-localization} has a unique solution satisfying \(\boldsymbol y(T_0)\in\overline{\Omega_1}\) and \(0\le T_0\le\min_i T_{2,\boldsymbol x_i}\).
\end{remark}

\begin{remark}\label{rem:joint-moving-source-reconstruction}
	Algorithm~\ref{alg:trajectory-parameter-reconstruction} first reconstructs the trajectory, then the ``source strength'' \(q\), and finally the scattering parameter \(\alpha\). The same procedure also provides a framework for jointly recovering the trajectory and strength of a moving point source in a homogeneous medium. In this setting, the input \(U_{\boldsymbol x_i}(t)\) is replaced by \(u^{\mathrm{in}}(\boldsymbol x_i(t),t)\). If emission starts at the known time \(0\), Step~1 yields the first arrival times \(t_i\). Furthermore, since \(|\boldsymbol x_i(t_i)-\boldsymbol z(0)|=ct_i\), Step~2 can be replaced by the four-distance localization procedure in~\cite[Section 2, Step (3)]{LiHuZhao2025} to recover the initial position \(\boldsymbol z(0)\) using the receiver positions \(\boldsymbol x_i(t_i)\). For \(t>0\), the subsequent Steps~3 and 4 apply, with \((\boldsymbol y,q)\) replaced by \((\boldsymbol z,\varphi)\). Stopping after Step~4 yields the jointly reconstructed trajectory and source strength.
\end{remark}

\section{Numerical examples}
\label{sec:numerical-examples}

In this section, we present four numerical experiments to assess the accuracy and robustness of the proposed reconstruction method for the trajectory and the time-dependent scattering parameter. 
We introduce multiplicative noise into the measured scattered fields:
\[
\widetilde U_{\boldsymbol{x}_i}(t)
=U_{\boldsymbol{x}_i}(t)\bigl(1+\varepsilon\eta_i(t)\bigr),
\qquad i=1,2,3,4,
\]
where \(\varepsilon\geq0\) is the noise level and \(\eta_i\) is obtained by interpolating independent samples from the uniform distribution on \([-1,1]\). 

The reconstruction errors are measured by the maximum absolute errors over the inversion interval \(I\):
\[
\mathrm{Err}_y=\sup_{t\in I}
\|\boldsymbol y(t)-\widehat{\boldsymbol y}(t)\|_\infty,
\qquad
\mathrm{Err}_\alpha=\sup_{t\in I}
|\alpha(t)-\widehat\alpha(t)|,
\]
where \(\widehat{\boldsymbol y}\) and \(\widehat\alpha\) denote the reconstructed trajectory and scattering parameter, respectively, and \(\|\cdot\|_\infty\) is the maximum norm in \(\mathbb R^3\).

To account for numerical errors, we implement Steps~3 and 5 of Algorithm~\ref{alg:trajectory-parameter-reconstruction} as follows.

To solve~\eqref{eq:trajectory-ode} in Step~3, we choose a small offset \(\delta>0\) and initialize the adaptive Runge--Kutta method RK45 at
\[
\boldsymbol{y}(T_0+\delta)
=\boldsymbol{y}_0+\delta cPA^{-1}\boldsymbol{1}.
\]
The trajectory is then reconstructed on \([T_0+\delta,T]\). On the initial interval, we use the first-order approximation
\[
\boldsymbol{y}(t)
=\boldsymbol{y}(T_0)
+(t-T_0)cPA^{-1}\boldsymbol{1},
\qquad T_0\leq t<T_0+\delta.
\]

In Step~5,~\eqref{eq:response-law} at \(T_0\) gives 
\[
p(T_0)=\gamma(T_0)q(T_0)=0,\qquad
p'(T_0+)
=\gamma(T_0)q'(T_0+)
=\frac{4\pi c}{\gamma(T_0)}
u^{\mathrm{in}}(\boldsymbol{y}(T_0),T_0+).
\]
We incorporate this derivative into a cubic spline interpolating the recovered values of \(p\). The derivative \(p'\) is approximated by analytically differentiating the spline. Near the left endpoint \(t_1=T_0+\delta\),
we recover \(\alpha\) by fitting the integral equation
\[
p(t)-p(t_1)
+\int_{t_1}^{t}\frac{4\pi c}{\gamma(s)^2}\alpha(s)p(s)\,ds
=
\int_{t_1}^{t}\frac{4\pi c}{\gamma(s)}
u^{\mathrm{in}}(\boldsymbol{y}(s),s)\,ds
\]
over a short time window. Away from the left endpoint, we recover
\(\alpha\) by using~\eqref{eq:alpha-recovery}. This treatment can reduce the differentiation error near \(T_0\).

\FloatBarrier
\subsection{Linear trajectory along the \texorpdfstring{$x$}{x}-axis}

In the first example, the wave speed is $c = 340$ and the scattering parameter is $\alpha(t)\equiv 0.0015$. The point-like scatterer moves at a constant velocity along the $x$-axis,
\[
\boldsymbol{y}(t) =
\begin{pmatrix}
10\,t \\
0 \\
0
\end{pmatrix}.
\]
The emitter is held stationary at
\[
\boldsymbol{z}(t) =
\begin{pmatrix}
1500 \\
0 \\
0
\end{pmatrix},
\]
and its strength varies linearly,
\[
\varphi(t) = 1 + 0.16\,t.
\]
Four stationary receivers are placed at the positions
\[
\boldsymbol{x}_1 = (1500, 0, 0)^{\top},\quad
\boldsymbol{x}_2 = (-1500, 0, 0)^{\top},\quad
\boldsymbol{x}_3 = (0, -1500, 0)^{\top},\quad
\boldsymbol{x}_4 = (0, 0, 1500)^{\top}.
\]
The initial scattering time is $T_0\approx 4.285714$. The inversion interval is $I=[T_0+0.002,T_0+5]$. 

\begin{table}[!ht]
    \centering
    \caption{Maximum absolute errors in the trajectory and scattering parameter for the linear trajectory at different noise levels.}
    \label{tab:case-1-errors}

    \begin{tabular}{ccccccc}
        \hline
        $\varepsilon$ & $0$ & $0.01$ & $0.02$ & $0.04$ & $0.08$ & $0.12$ \\
        \hline

        $\mathrm{Err}_y$ & $4.31\times10^{-7}$ & $1.818$ & $3.640$ & $7.302$ & $14.646$ & $21.928$ \\
        $\mathrm{Err}_\alpha$ & $4.30\times10^{-9}$ & $7.36\times10^{-5}$ & $1.49\times10^{-4}$ & $3.05\times10^{-4}$ & $6.44\times10^{-4}$ & $1.03\times10^{-3}$ \\
        \hline
    \end{tabular}

\end{table}

The reconstruction errors under increasing noise levels are reported in Table~\ref{tab:case-1-errors}, and the reconstructed first-coordinate trajectories are plotted against the true trajectory in Figure~\ref{fig:case-1-trajectory}. The trajectory error grows approximately linearly with the noise level, from $4.31\times10^{-7}$ in the noise-free case to $21.928$ at $\varepsilon=0.12$, indicating gradual degradation under measurement perturbations. The reconstruction of the constant scattering parameter is shown in Figure~\ref{fig:case-1-alpha}. Its maximum absolute error increases from $4.30\times10^{-9}$ without noise to $7.36\times10^{-5}$ at $\varepsilon=0.01$ and $1.03\times10^{-3}$ at $\varepsilon=0.12$, with the largest deviations in the noise-free case occurring near the initial time.

\begin{figure}[p]
\begin{minipage}[t][0.48\textheight][t]{\textwidth}
\centering
\includegraphics[width=\textwidth,height=0.38\textheight,keepaspectratio]{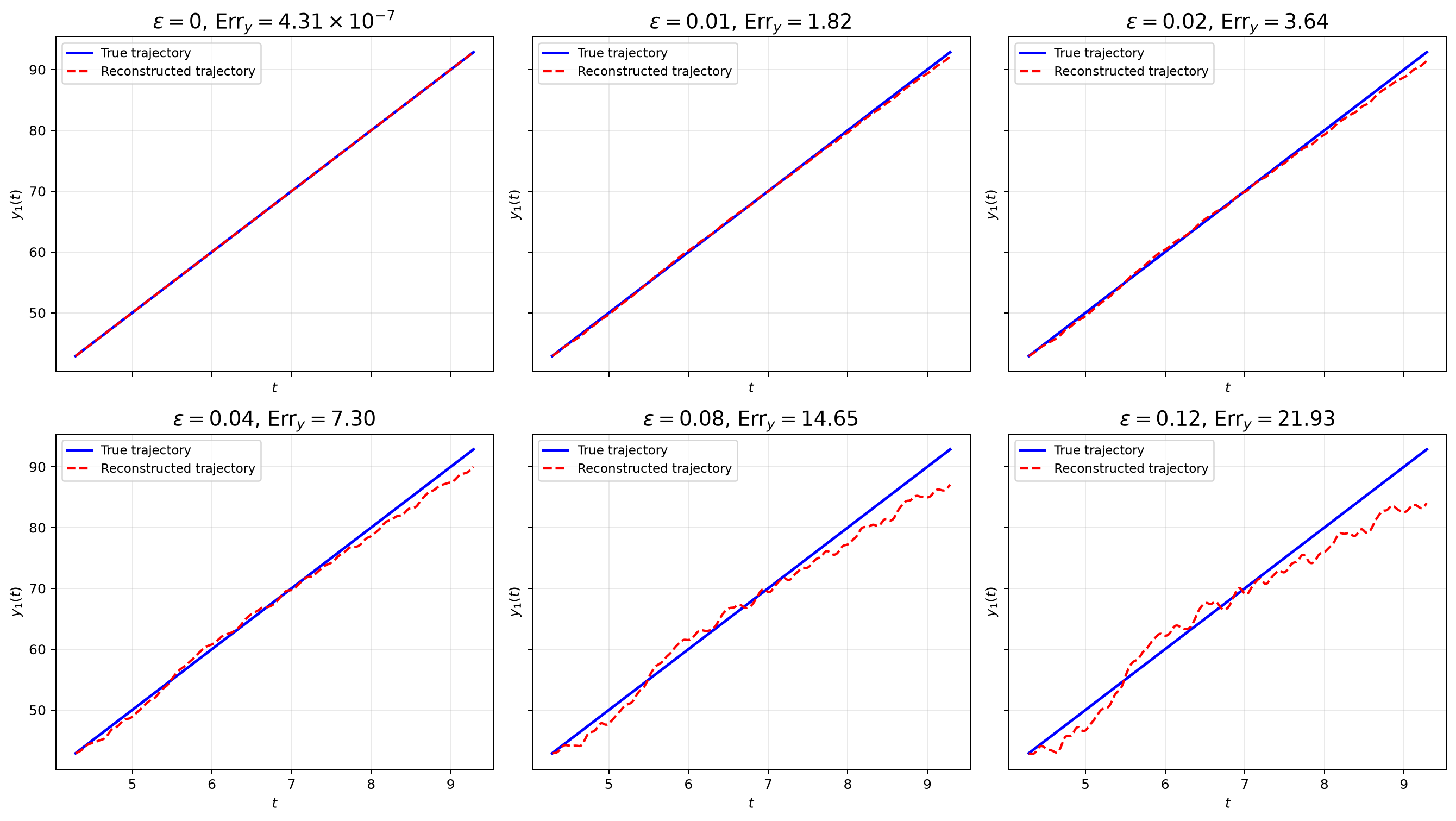}
\caption{Reconstruction of the first coordinate of a linearly moving scatterer. Blue curves represent the true trajectory and red dashed curves represent the reconstructed trajectory. Results are shown at noise levels $\varepsilon = 0, 0.01, 0.02, 0.04, 0.08, 0.12$.}
\label{fig:case-1-trajectory}
\end{minipage}
\par\vspace{0.02\textheight}
\begin{minipage}[t][0.48\textheight][t]{\textwidth}
\centering
\includegraphics[width=\textwidth,height=0.38\textheight,keepaspectratio]{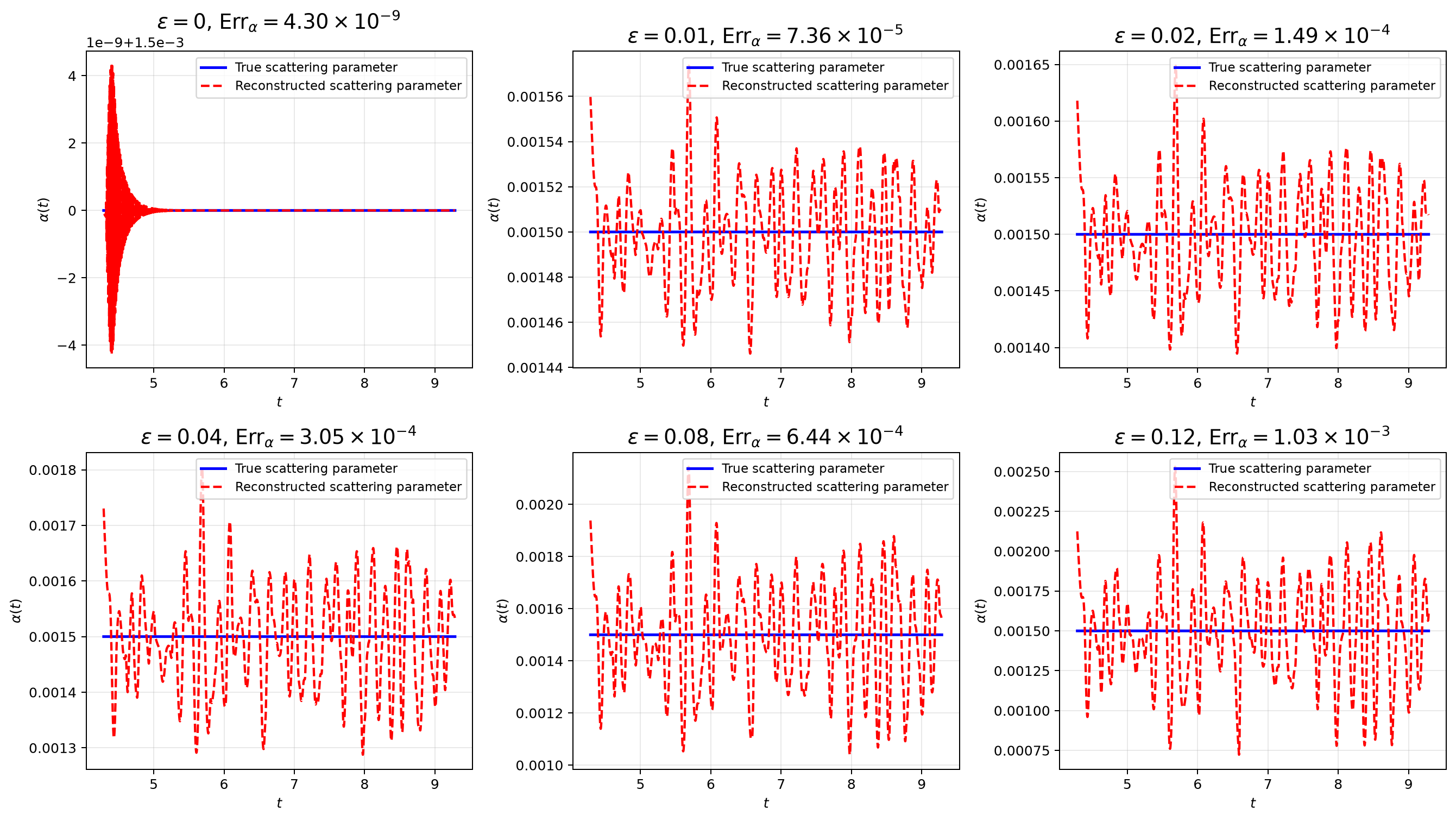}
\caption{Reconstruction of the constant scattering parameter for a linearly moving scatterer. Blue curves represent the true scattering parameter and red dashed curves represent the reconstructed scattering parameter. Results are shown at noise levels $\varepsilon = 0, 0.01, 0.02, 0.04, 0.08, 0.12$.}
\label{fig:case-1-alpha}
\end{minipage}
\end{figure}

\subsection{Nonlinear planar trajectory}

In the second example, the wave speed is fixed at $c = 340$. The scatterer follows the nonlinear planar path
\[
\boldsymbol{y}(t) =
\begin{pmatrix}
9\,t \\
45\sin(0.85t) \\
0
\end{pmatrix}.
\]
The emitter remains at $x=1500$ and oscillates in the $z$-direction:
\[
\boldsymbol{z}(t) =
\begin{pmatrix}
1500 \\
0 \\
8\cos(0.45t)
\end{pmatrix}.
\]
Its strength is a positive Gaussian modulation,
\[
\varphi(t) = 1.2 + 0.35\,\exp\!\left(-\frac{(t - 4)^2}{2(2.2)^2}\right).
\]
The four receivers remain stationary at
\[
\boldsymbol{x}_1 = (1500, 0, 0)^{\top},\quad
\boldsymbol{x}_2 = (-1500, 0, 0)^{\top},\quad
\boldsymbol{x}_3 = (0, -1500, 0)^{\top},\quad
\boldsymbol{x}_4 = (0, 0, 1500)^{\top}.
\]
The scattering parameter is $\alpha(t)=0.0005+0.00025(t-T_0)$ for $t\geq T_0$. The initial scattering time is $T_0\approx 4.298534$. The inversion interval is
\[
I=[T_0+0.002,T_0+10].
\]

\begin{table}[!ht]
    \centering
    \caption{Maximum absolute errors in the trajectory and scattering parameter for the nonlinear planar trajectory at different noise levels.}
    \label{tab:case-2-errors}

    \begin{tabular}{ccccccc}
        \hline
        $\varepsilon$ & $0$ & $0.01$ & $0.02$ & $0.04$ & $0.08$ & $0.10$ \\
        \hline

        $\mathrm{Err}_y$ & $3.43\times10^{-5}$ & $3.308$ & $6.601$ & $13.193$ & $26.602$ & $33.461$ \\
        $\mathrm{Err}_\alpha$ & $3.40\times10^{-10}$ & $6.37\times10^{-5}$ & $1.30\times10^{-4}$ & $2.73\times10^{-4}$ & $6.00\times10^{-4}$ & $7.86\times10^{-4}$ \\
        \hline
    \end{tabular}

\end{table}

The reconstruction errors are summarized in Table~\ref{tab:case-2-errors}, and the reconstructed orbits in the $xy$-plane are compared with the true orbit in Figure~\ref{fig:case-2-trajectory}. In the noise-free case, the reconstruction closely follows the true trajectory, with a maximum absolute error of $3.43\times10^{-5}$. At $\varepsilon=0.01$, the main oscillatory structure remains accurately resolved and the error is $3.308$. As the noise level increases, the error grows to $13.193$ at $\varepsilon=0.04$ and $33.461$ at $\varepsilon=0.10$. The reconstruction retains the principal planar oscillation at low-to-moderate noise levels, whereas pronounced local deviations appear at the two highest noise levels. Figure~\ref{fig:case-2-alpha} shows that the linear variation of the scattering parameter is also recovered, with maximum absolute errors of $3.40\times10^{-10}$ without noise, $6.37\times10^{-5}$ at $\varepsilon=0.01$, and $7.86\times10^{-4}$ at $\varepsilon=0.10$.

\begin{figure}[p]
\begin{minipage}[t][0.48\textheight][t]{\textwidth}
\centering
\includegraphics[width=\textwidth,height=0.38\textheight,keepaspectratio]{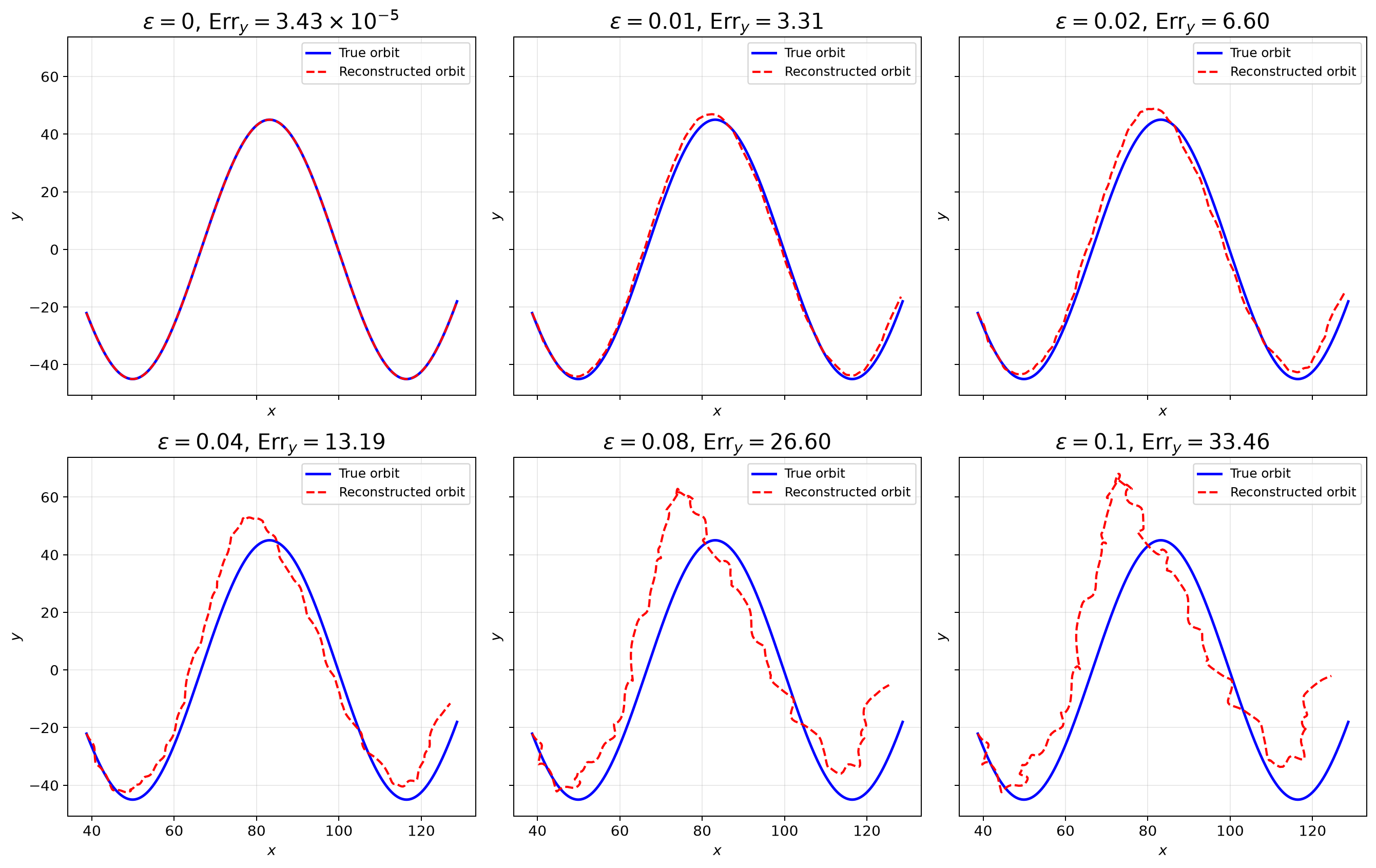}
\caption{Reconstruction of a nonlinearly moving scatterer in the $xy$-plane. Blue curves represent the true orbit and red dashed curves represent the reconstructed orbit. Results are shown for noise levels $\varepsilon = 0, 0.01, 0.02, 0.04, 0.08, 0.10$.}
\label{fig:case-2-trajectory}
\end{minipage}
\par\vspace{0.02\textheight}
\begin{minipage}[t][0.48\textheight][t]{\textwidth}
\centering
\includegraphics[width=\textwidth,height=0.38\textheight,keepaspectratio]{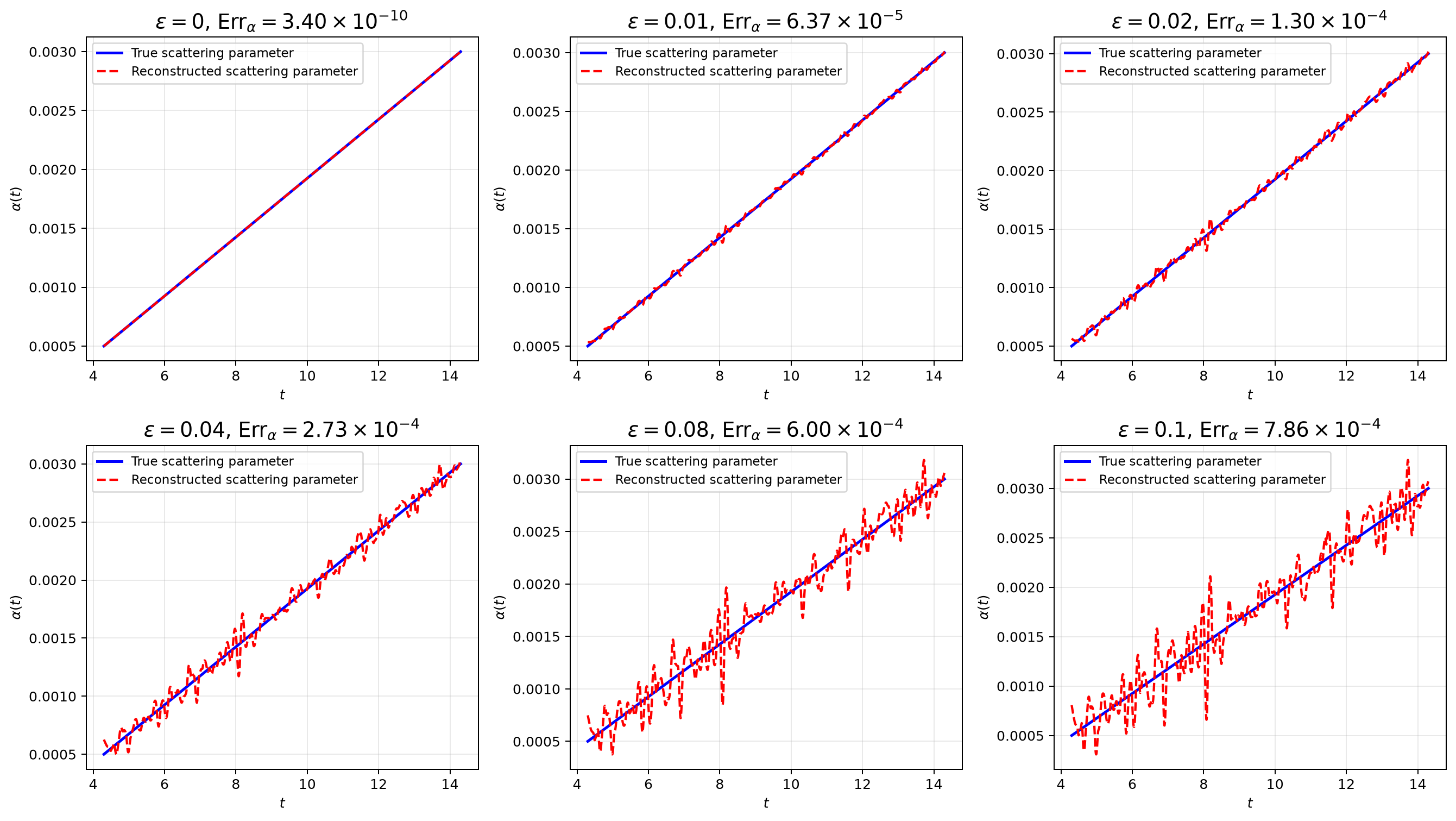}
\caption{Reconstruction of the linearly varying scattering parameter for a nonlinearly moving scatterer in the $xy$-plane. Blue curves represent the true scattering parameter and red dashed curves represent the reconstructed scattering parameter. Results are shown for noise levels $\varepsilon = 0, 0.01, 0.02, 0.04, 0.08, 0.10$.}
\label{fig:case-2-alpha}
\end{minipage}
\end{figure}

\subsection{Fully three-dimensional helical trajectory}

In the third example, the wave speed is $c = 340$. The scatterer executes a helical motion in $\mathbb{R}^3$,
\[
\boldsymbol{y}(t) =
\begin{pmatrix}
10\,t \\
80\cos(1.5\,t) \\
80\sin(1.5\,t)
\end{pmatrix},
\]
while the emitter moves along an oscillatory path,
\[
\boldsymbol{z}(t) =
\begin{pmatrix}
1500 + 20\cos(0.5\,t) \\
0 \\
30\sin(0.6\,t)
\end{pmatrix},
\]
with a mildly oscillating strength,
\[
\varphi(t) = 1 + 0.01\sin(t).
\]
All four receivers are also in motion:
\[
\boldsymbol{x}_1(t) =
\begin{pmatrix}
1500 + 20\cos(0.8\,t) \\
0 \\
0
\end{pmatrix},\quad
\boldsymbol{x}_2(t) =
\begin{pmatrix}
0 \\
1500 + 20\sin(0.7\,t) \\
0
\end{pmatrix},
\]
\[
\boldsymbol{x}_3(t) =
\begin{pmatrix}
0 \\
-1500 - 20\cos(0.6\,t) \\
0
\end{pmatrix},\quad
\boldsymbol{x}_4(t) =
\begin{pmatrix}
0 \\
0 \\
1500 + 20\cos(0.9\,t)
\end{pmatrix}.
\]
The scattering parameter is $\alpha(t)=0.0015+0.001\sin(1.2(t-T_0))$ for $t\geq T_0$. The initial scattering time is $T_0\approx 4.349045$. The inversion interval is $I=[T_0+0.002,T_0+10]$. In this example, the emitter, scatterer, and receivers all move. The helical trajectory has radius $80$ and a maximal speed of approximately $120.4$, which remains strictly below $c=340$.

\begin{table}[!ht]
    \centering
    \caption{Maximum absolute errors in the trajectory and scattering parameter for the helical trajectory at different noise levels.}
    \label{tab:case-3-errors}

    \begin{tabular}{ccccccc}
        \hline
        $\varepsilon$ & $0$ & $0.01$ & $0.02$ & $0.04$ & $0.08$ & $0.10$ \\
        \hline

        $\mathrm{Err}_y$ & $3.37\times10^{-4}$ & $4.659$ & $9.330$ & $18.629$ & $36.671$ & $45.688$ \\
        $\mathrm{Err}_\alpha$ & $6.35\times10^{-6}$ & $1.10\times10^{-4}$ & $2.21\times10^{-4}$ & $4.43\times10^{-4}$ & $1.11\times10^{-3}$ & $1.55\times10^{-3}$ \\
        \hline
    \end{tabular}

\end{table}

The reconstruction errors are listed in Table~\ref{tab:case-3-errors}, and the three-dimensional reconstructed orbits are shown together with the true orbits in Figure~\ref{fig:case-3-trajectory}. At $\varepsilon=0$, the trajectory error is $3.37\times10^{-4}$, rising to $45.688$ at $\varepsilon=0.10$. Despite the substantial degradation at the highest noise levels, the overall helical structure remains discernible in all reconstructions. The periodic scattering parameter is reconstructed in Figure~\ref{fig:case-3-alpha}, with its oscillatory trend retained as the noise level increases. Its maximum absolute error grows from $6.35\times10^{-6}$ without noise to $1.10\times10^{-4}$ at $\varepsilon=0.01$ and $1.55\times10^{-3}$ at $\varepsilon=0.10$.

\begin{figure}[p]
\begin{minipage}[t][0.48\textheight][t]{\textwidth}
\centering
\includegraphics[width=\textwidth,height=0.38\textheight,keepaspectratio]{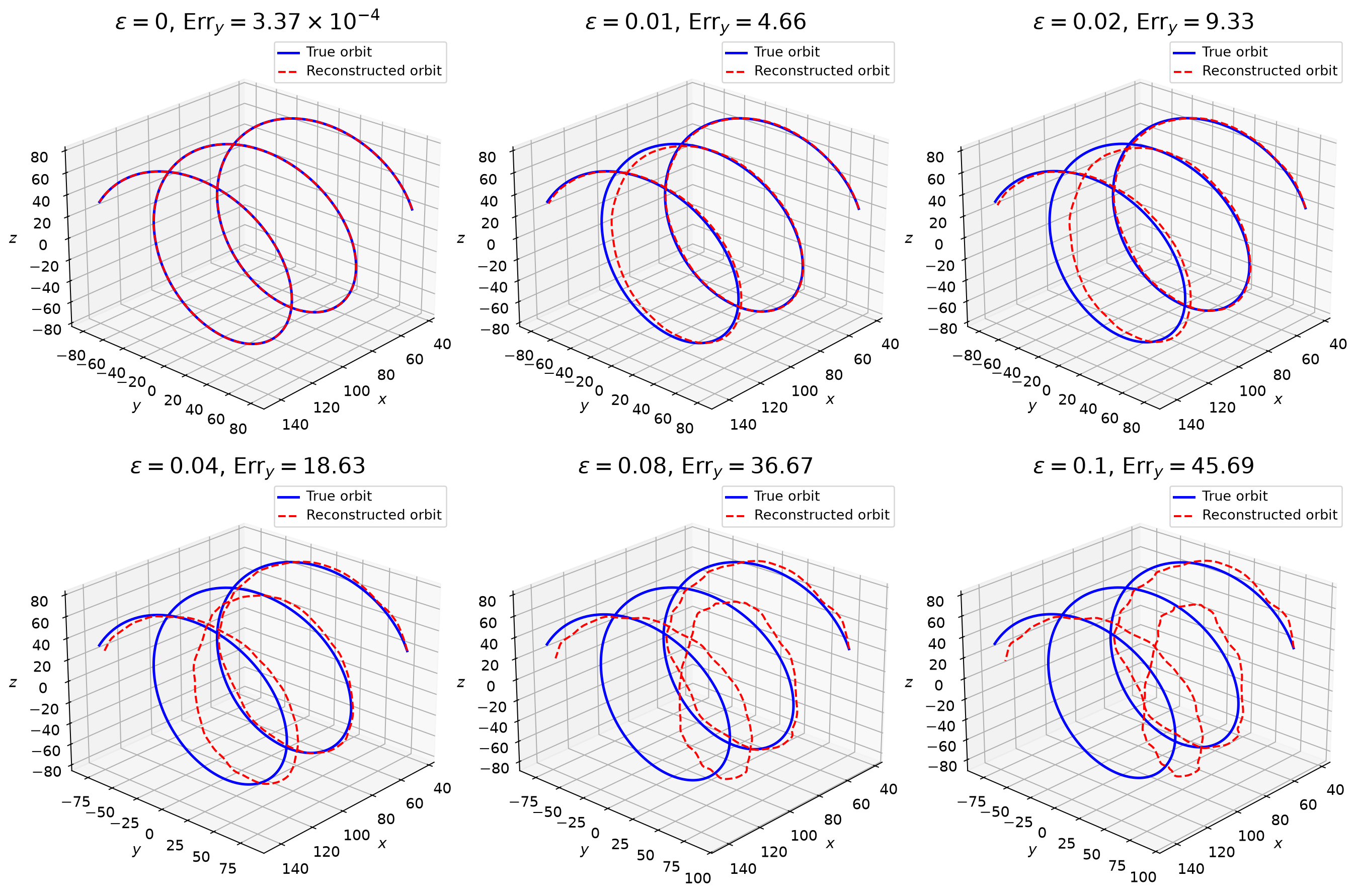}
\caption{Reconstruction of a helically moving scatterer with moving emitter and receivers. Blue curves represent the true orbit and red dashed curves represent the reconstructed orbit. Results are shown at noise levels $\varepsilon = 0, 0.01, 0.02, 0.04, 0.08, 0.10$.}
\label{fig:case-3-trajectory}
\end{minipage}
\par\vspace{0.02\textheight}
\begin{minipage}[t][0.48\textheight][t]{\textwidth}
\centering
\includegraphics[width=\textwidth,height=0.38\textheight,keepaspectratio]{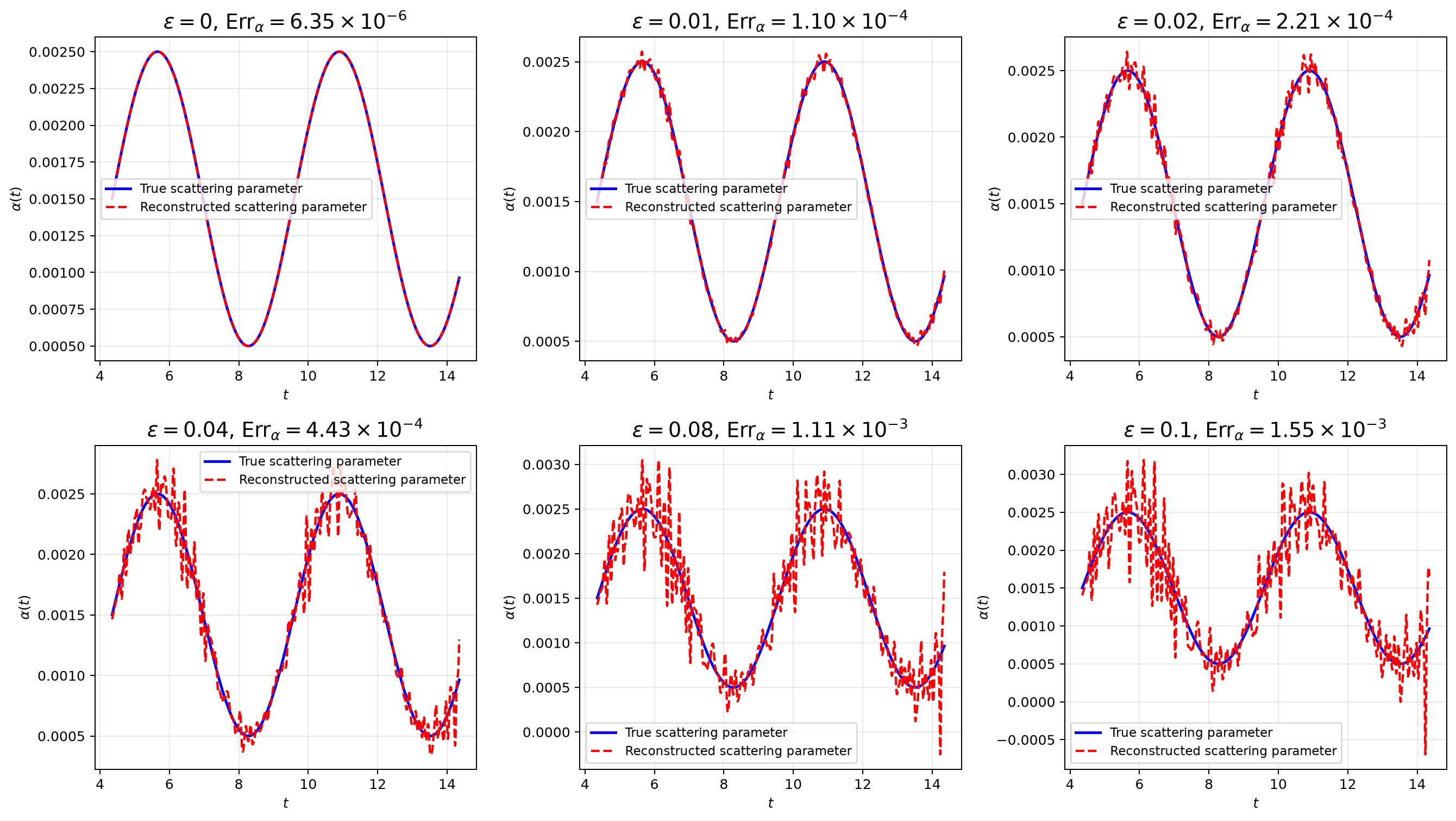}
\caption{Reconstruction of the periodically varying scattering parameter for a helically moving scatterer with moving emitter and receivers. Blue curves represent the true scattering parameter and red dashed curves represent the reconstructed scattering parameter. Results are shown at noise levels $\varepsilon = 0, 0.01, 0.02, 0.04, 0.08, 0.10$.}
\label{fig:case-3-alpha}
\end{minipage}
\end{figure}

\subsection{Effect of the wave speed on reconstruction accuracy}

In the final example, we investigate how the reconstruction accuracy depends on the wave speed $c$. To compare the same physical portion of the orbit for every $c$, let $s=t-T_0(c)$ denote the elapsed time after the first incident arrival and set
\[
\boldsymbol{y}(t) =
\begin{pmatrix}
30\,s \\
8\sin(0.8\,s) \\
5\cos(0.7\,s)
\end{pmatrix},
\]
where $\sup_{t\geq0}\|\boldsymbol{y}'(t)\|\approx30.9$. Equivalently, for each wave speed the global-time trajectory is shifted so that the same reference point is reached at $T_0(c)$. The emitter is located near $(1500,0,0)$ with a small $z$-oscillation,
\[
\boldsymbol{z}(t) =
\begin{pmatrix}
1500 \\
0 \\
2\cos(0.5\,t)
\end{pmatrix},
\]
and its strength is always positive,
\[
\varphi(t) = 1.3 + 0.2\sin(0.5\,t) + 0.1\cos(0.8\,t) \in [1, 1.6].
\]
The four receivers are stationary at the same positions as in the first two examples,
\[
\boldsymbol{x}_1 = (1500, 0, 0)^{\top},\;
\boldsymbol{x}_2 = (-1500, 0, 0)^{\top},\;
\boldsymbol{x}_3 = (0, -1500, 0)^{\top},\;
\boldsymbol{x}_4 = (0, 0, 1500)^{\top}.
\]
The scattering parameter is given by
\[
\begin{aligned}
\alpha(t)
&=0.0006+0.0014\exp\!\left[-\left(\frac{s-2}{0.65}\right)^2\right]
+0.0010\exp\!\left[-\left(\frac{s-7}{1.1}\right)^2\right]\\
&\quad+\frac{0.0008}{1+\exp[-1.5(s-4.5)]}.
\end{aligned}
\]
The noise level is fixed at $\varepsilon=0.01$. For each of the six wave speeds, we perform 50 reconstructions using the same set of 50 random seeds to generate the measurement noise. The wave speed varies over $\{80,180,250,340,400,650\}$ and the inversion interval corresponds to $s\in[0.002,10]$. Even the smallest speed, $c=80$, exceeds the maximal scatterer speed by a factor of approximately $2.6$ and therefore satisfies the subsonic condition.

\begin{table}[!ht]
    \centering
    \caption{Means and sample variances of the maximum absolute errors over 50 reconstructions for each wave speed $c$ ($\varepsilon = 0.01$).}
    \label{tab:case-4-errors}
    \begin{tabular}{@{}c*{6}{@{\hspace{5.5pt}}c}@{}}
        \hline
        $c$ & $80$ & $180$ & $250$ & $340$ & $400$ & $650$ \\
        \hline

        $\operatorname{Mean}(\mathrm{Err}_y)$ & $1.603$ & $2.549$ & $2.823$ & $3.352$ & $3.788$ & $4.758$ \\
        $\operatorname{Var}(\mathrm{Err}_y)$ & $0.494$ & $1.099$ & $1.456$ & $1.404$ & $1.612$ & $1.670$ \\
        $\operatorname{Mean}(\mathrm{Err}_\alpha)$ & $1.48\times10^{-4}$ & $7.88\times10^{-5}$ & $6.57\times10^{-5}$ & $5.47\times10^{-5}$ & $5.01\times10^{-5}$ & $3.70\times10^{-5}$ \\
        $\operatorname{Var}(\mathrm{Err}_\alpha)$ & $4.80\times10^{-9}$ & $3.57\times10^{-10}$ & $1.42\times10^{-10}$ & $3.64\times10^{-11}$ & $2.90\times10^{-11}$ & $1.54\times10^{-11}$ \\
        \hline
    \end{tabular}

\end{table}

The means and sample variances of the reconstruction errors are reported in Table~\ref{tab:case-4-errors}, and the trajectories are visualized in Figure~\ref{fig:case-4-trajectory}. At $c=80$, the mean trajectory error is $1.603$. As $c$ increases, the mean error rises to $2.549$ at $c=180$, $2.823$ at $c=250$, and $3.352$ at $c=340$, reaching $4.758$ at $c=650$. In contrast, the reconstruction of the time-dependent scattering parameter in Figure~\ref{fig:case-4-alpha} improves with increasing wave speed: its mean maximum absolute error decreases from $1.48\times10^{-4}$ at $c=80$ to $5.47\times10^{-5}$ at $c=340$ and $3.70\times10^{-5}$ at $c=650$. Thus, for this configuration and the 50 noise realizations, the mean trajectory and scattering-parameter errors exhibit different trends as the wave speed varies.

\begin{figure}[p]
\begin{minipage}[t][0.48\textheight][t]{\textwidth}
\centering
\includegraphics[width=\textwidth,height=0.38\textheight,keepaspectratio]{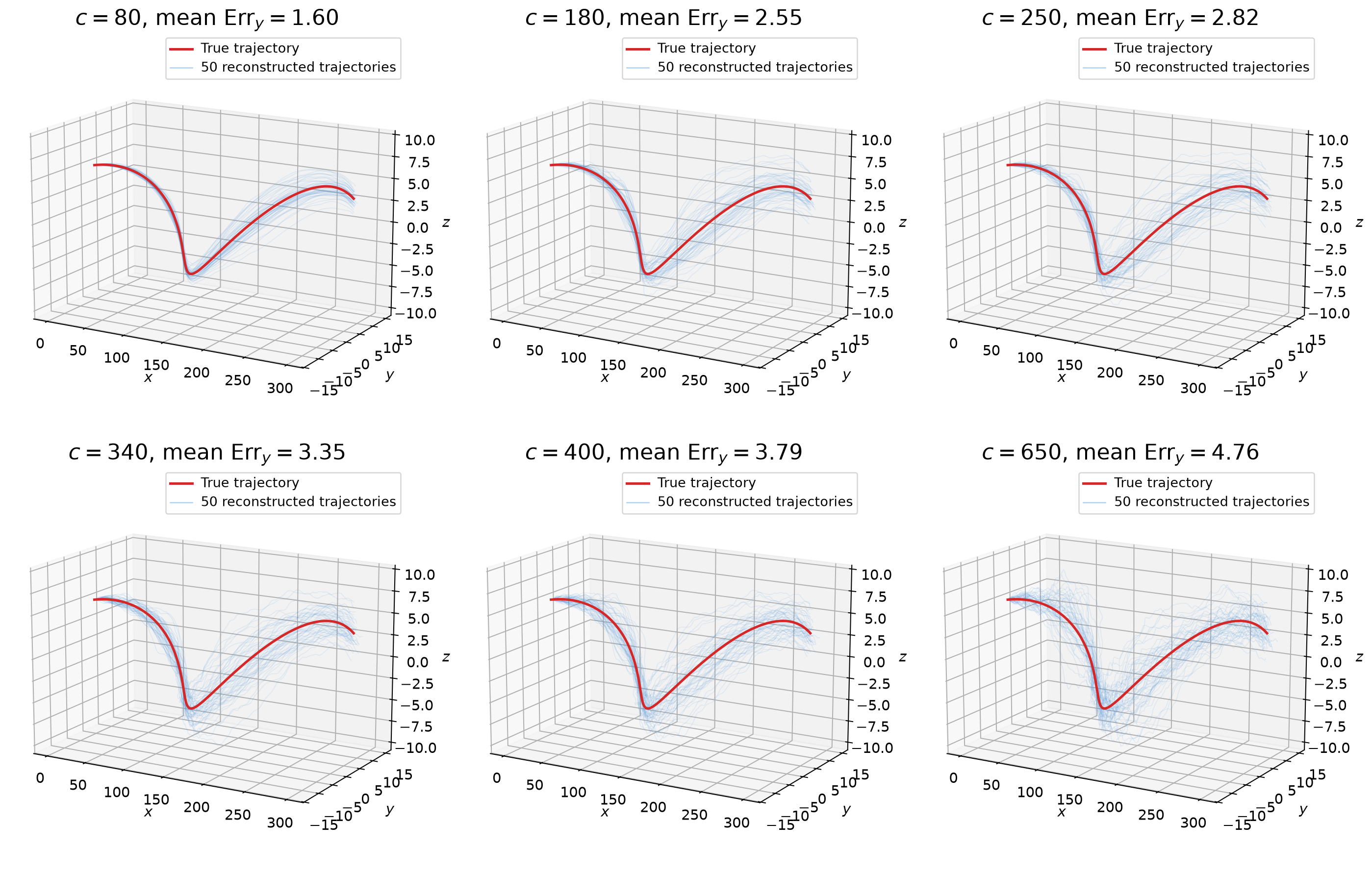}
\caption{Reconstruction of the same elapsed-time segment of the scatterer trajectory for $c = 80, 180, 250, 340, 400, 650$ at noise level $\varepsilon=0.01$. Red curves represent the true orbit and light blue curves represent the 50 reconstructed orbits.}
\label{fig:case-4-trajectory}
\end{minipage}
\par\vspace{0.02\textheight}
\begin{minipage}[t][0.48\textheight][t]{\textwidth}
\centering
\includegraphics[width=\textwidth,height=0.38\textheight,keepaspectratio]{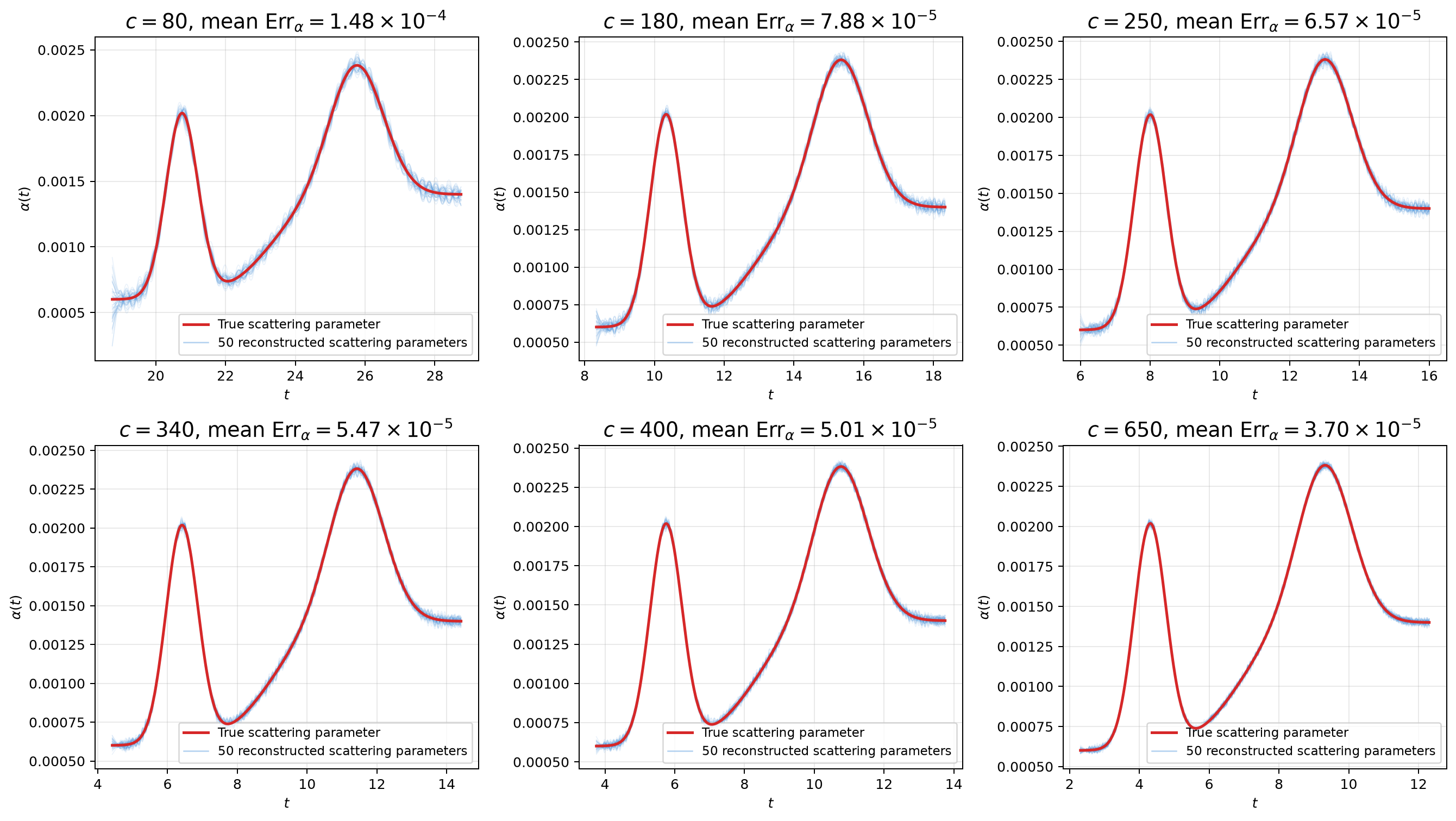}
\caption{Reconstruction of the time-dependent scattering parameter over the same elapsed-time interval for $c = 80, 180, 250, 340, 400, 650$ at noise level $\varepsilon=0.01$. Red curves represent the true scattering parameter and light blue curves represent the 50 reconstructed scattering parameters.}
\label{fig:case-4-alpha}
\end{minipage}
\end{figure}

Across the four numerical examples, the proposed inversion algorithm recovers both the scatterer trajectory and the time-dependent scattering parameter. In the first three examples, the noise-free trajectory errors range from \(4.31\times10^{-7}\) to \(3.37\times10^{-4}\), while the scattering-parameter errors range from \(3.40\times10^{-10}\) to \(6.35\times10^{-6}\). Both errors increase with the noise level, although the principal trajectory shapes and parameter trends remain identifiable. The larger errors in \(\alpha\) near the initial time arise from numerical differentiation and division by the small quantity \(p(t)\) in~\eqref{eq:alpha-recovery}. The wave-speed study shows different trends for the two reconstructed quantities: the mean trajectory error increases with \(c\), whereas the mean scattering-parameter error decreases. These experimental results corroborate the effectiveness and reliability of the inversion algorithm under the stated assumptions.

\section{Conclusion}
\label{sec:conclusion}

In this work, we have investigated an inverse scattering problem for the scalar wave equation in a fully dynamic configuration in which the point emitter, a point-like scatterer, and multiple receivers are all moving. We established a scattering model for the point-like scatterer, as a limit of a short-range potential model concentrated around the moving trajectory, providing a solid mathematical foundation for the subsequent inverse analysis. For the inverse problem, we reduce the reconstruction to a system of nonlinear ordinary differential equations. Four noncoplanar receivers are sufficient to achieve effective recovery. Numerical experiments demonstrate the effectiveness and robustness of the method under various dynamic configurations. Future work includes stability analysis with respect to measurement errors and extensions to more general scattering models, such as multiple point scatterers or extended obstacles.

\section*{Acknowledgments}
Part of this work was completed during M.~Li's visit to City University of
Hong Kong from August to October 2026. 
The hospitiality of the CityU was greatly appreciated. 

\end{document}